\documentclass[10pt,a4paper]{amsart}

\usepackage{amsmath}
\usepackage{amsthm}
\usepackage{enumerate} 
\usepackage{amssymb}
\usepackage{graphicx}
\usepackage[all]{xy}
\usepackage{hyperref}
\usepackage{aliascnt}
\usepackage{stmaryrd} 
\usepackage{mathtools}
\usepackage{verbatim} 
\usepackage{txfonts} 
\usepackage{lmodern}
\usepackage{setspace}   

\newtheorem{lma}{Lemma}[section]

\newaliascnt{thmCt}{lma}
\newtheorem{thm}[thmCt]{Theorem}
\aliascntresetthe{thmCt}

\newaliascnt{corCt}{lma}
\newtheorem{cor}[corCt]{Corollary}
\aliascntresetthe{corCt}

\newaliascnt{propCt}{lma}
\newtheorem{prop}[propCt]{Proposition}
\aliascntresetthe{propCt}

\newtheorem*{thm*}{Theorem}
\newtheorem*{cor*}{Corollary}
\newtheorem*{prop*}{Proposition}

\theoremstyle{definition}

\newaliascnt{prgCt}{lma}
\newtheorem{prg}[prgCt]{}
\aliascntresetthe{prgCt}

\newaliascnt{dfnCt}{lma}
\newtheorem{dfn}[dfnCt]{Definition}
\aliascntresetthe{dfnCt}

\newaliascnt{rmkCt}{lma}
\newtheorem{rmk}[rmkCt]{Remark}
\aliascntresetthe{rmkCt}

\newaliascnt{rmksCt}{lma}

\aliascntresetthe{rmksCt}

\newaliascnt{ntnCt}{lma}

\aliascntresetthe{ntnCt}

\newaliascnt{ntnsCt}{lma}

\aliascntresetthe{ntnsCt}

\newaliascnt{qstCt}{lma}

\aliascntresetthe{qstCt}

\newaliascnt{prblCt}{lma}

\aliascntresetthe{prblCt}

\newaliascnt{exaCt}{lma}

\aliascntresetthe{exaCt}

\newcommand{\M}{\mathrm{M}}

\newcommand{\T}{\mathbb{T}}
\newcommand{\N}{\mathbb{N}}

\newcommand{\R}{\mathbb{R}}
\newcommand{\Z}{\mathbb{Z}}
\newcommand{\K}{\mathrm{K}}

\DeclareMathOperator{\rank}{rank}

\DeclareMathOperator{\card}{card}

\DeclareMathOperator{\spectrum}{sp}

\newcommand{\CatCa}{C^*}
\newcommand{\CatPoM}{\mathrm{PoM}}

\DeclareMathOperator{\Cu}{Cu}

\DeclareMathOperator{\NCCW}{NCCW}
\DeclareMathOperator{\CW}{CW}
\DeclareMathOperator{\AI}{AI}
\DeclareMathOperator{\A}{A}
\DeclareMathOperator{\UHF}{UHF}

\DeclareMathOperator{\AF}{AF}
\DeclareMathOperator{\Lsc}{Lsc}

\DeclareMathOperator{\supp}{supp}
\DeclareMathOperator{\Hom}{Hom}
\DeclareMathOperator{\id}{id}

\begin{document}
\onehalfspacing
\title{Uniformly Based Cuntz semigroups and approximate intertwinings}
\author{Laurent Cantier}

\address{Laurent Cantier,
Departament de Matem\`{a}tiques \\
Universitat Aut\`{o}noma de Barcelona \\
08193 Bellaterra, Barcelona, Spain
}
\email[]{lcantier@mat.uab.cat}

\thanks{The author was supported by MINECO through the grant BES-2016-077192-POP and partially supported by the grants MDM-2014-0445 and MTM-2017-83487 at the Centre de Recerca Matem\`atica in Barcelona.}

\begin{abstract}
We study topological aspects of the category of abstract Cuntz semigroups, termed $\Cu$. We provide a suitable setting in which we are able to uniformly control how to approach an element of a $\Cu$-semigroup by a rapidly increasing sequence. This approximation induces a semimetric on the set of $\Cu$-morphisms, generalizing $\Cu$-metrics that had been constructed in the past for some particular cases. Further, we develop an approximate intertwining theory for the category $\Cu$. Finally, we give several applications such as the classification of unitary elements of any unital $\AF$-algebra by means of the functor $\Cu$.
\end{abstract}
\maketitle

\section{Introduction}
In the past decade, the Cuntz semigroup has emerged as a valuable asset for classification of both simple and non-simple $\CatCa$-algebras. The regain of interest for this invariant started around 15 years ago, when Toms exhibited a counter-example to the original Elliot conjecture using precisely the Cuntz semigroup, proving in the process that within this semigroup, lies crucial information for classification of $\CatCa$-algebras; see \cite{T08}. Nevertheless, the original construction introduced by Cuntz in the late 70's (see \cite{C78}), was not well-suited for classification of $\CatCa$-algebras. More precisely, it had already been shown that the original Cuntz semigroup did not preserve inductive limits of $\CatCa$-algebras. Back to the 00's, a short while after Toms counter-example, Coward, Elliott and Ivanescu presented a new version of the Cuntz semigroup, in \cite{CEI08}, together with an abstract category based on order-theoretical axioms, termed $\Cu$. This completed version of the Cuntz semigroup, referred nowadays to as the Cuntz semigroup, is indeed a more adapted invariant since it defines a continuous functor $\Cu:\CatCa\longrightarrow \Cu$, as shown in \cite{APT14}.  Both the category $\Cu$ and the functor $\Cu$ have been studied extensively (e.g. in \cite{APS11}, \cite{APT14}, \cite{APT09}, \cite{CRS10}, \cite{RS19}, \cite{T19}) and the latter has appeared to be a powerful invariant. Let us recall some of the most notable results that have unraveled since then. First, in addition to being a continuous functor (that is, a functor that preserves inductive limits), the functor $\Cu$ also preserves ultraproducts (see \cite{APT19}). It has also been shown that the category $\Cu$ is closed (see \cite{APT20} and \cite{APT20b}), complete and cocomplete (see \cite{APT19}). A highly relevant result in the structure theory of Cuntz semigroups is \cite{APRT21}, where it is shown that the Cuntz semigroups of stable rank one $\CatCa$-algebras satisfy Riesz interpolation.
As far as the classification of $\CatCa$-algebras is concerned, it has been shown in \cite{ADPS14} that one can functorially recover the original Elliott invariant of any unital, simple, nuclear, finite, $\mathcal{Z}$-stable $\CatCa$-algebra $A$ from the Cuntz semigroup of $\mathcal{C}(\T)\otimes A$. Then, Robert showed in \cite{R12} that the functor $\Cu$ classifies $^*$-homomorphisms from any inductive limit of one-dimensional $\NCCW$ complexes with trivial $\K_1$-groups to any $\CatCa$-algebra of stable rank one. In particular, $\Cu$ is a complete invariant for the latter domain class. Note that the Cuntz semigroup, although being a key tool for the non-simple setting, is still not fully endorsed by all and has a few drawbacks: its computation is in general quite complex (not to say nearly impossible), added to the fact that it only contains tracial data and $\K_0$ information. To address this lack of $\K_1$ information, the author has been introducing and intensively studying a unitary version of the Cuntz semigroup in \cite{C20one} and \cite{C20two}, that seems promising towards future classification results of non-simple $\CatCa$-algebras beyond the trivial $\K_1$ setting. 

Overall, it is only natural that abstract Cuntz semigroups, often termed $\Cu$-semigroups, have become the target of further study. This paper tends to investigate more in depth topological aspects of such semigroups. We remark that these objects have an underlying algebraic structure, since they are positively ordered monoids, but they are also equipped with an order-theoretic topology, given by axioms (O1) and (O2) and built upon an auxiliary relation called the \emph{way-below} or the \emph{compact-containment} relation. For instance, (O1) is an analogue version of completeness. We recall that the compact-containment relation is entirely determined by the order of the monoid.

Topological aspects of $\Cu$-semigroups have already been exploited in the past. E.g. a metric on $\Hom_{\Cu}(\Lsc(]0,1],\overline{\N}),T)$ has been constructed in \cite{RS09} and a notion of countably-based $\Cu$-semigroups has been defined (as an analogue of separability). But at the time of writing, it seems there is still much to be explored and conceptualized. We point out that the starting point for this paper could be traced back to a paper by Gong, Jiang and Li (see  \cite{GJL20}) in which they introduce an extended version of the Elliott invariant. More particularly, they construct two $\CatCa$-algebras which were indistinguishable using the current Elliott invariant, based on $\K$-theoretical tools, but which are not isomorphic since they do not agree on the new ingredient they have added to this augmented invariant. They also claim that one could use tools developed on the Cuntz semigroup to conclude that the latter was also unable to differentiate these two $\CatCa$-algebras, constructed as inductive limits of one-dimensional $\NCCW$ complexes. To my knowledge, I would venture to say that very little has been done about an analogous version of an approximate intertwining theorem for the category $\Cu$. Therefore, it appeared to be relevant to dig in this direction. All the more so, given that the Cuntz semigroup has emerged as a noteworthy invariant for $\CatCa$-algebras and intertwinings theorems have always played a key role in proofs of classification theorems. 

\textbf{Organization of the paper.} In a first part, we define a topological notion of \emph{uniform basis} for $\Cu$-semigroups, that resembles -from afar- the notion of second-countability for topological spaces. We also establish an abstract notion of comparison of $\Cu$-morphisms on a given subset of the domain. Through the use of these two notions, we are able to construct a suitable setting in which the following theorem is proven for inductive sequences of uniformly based $\Cu$-semigroups.

\begin{thm}
Let $(S_i,\sigma_{ij})_{i\in\N}$ and $(T_i,\tau_{ij})_{i\in\N}$ be two inductive sequences of uniformly based $\Cu$-semigroups. Let $(S,\sigma_{i\infty})_i$ and $(T,\tau_{i\infty})_i$ be their respective inductive limits in $\Cu$. 
For any $i\in\N$, consider a uniform basis of $S_i$, respectively of $T_i$, that we both denote by $(M_n,\epsilon_n)_n$. (Without referring to the index $i$ to ease notations and the basis we refer to is always clear.)

Suppose there exist $\Cu$-morphisms $c_i$ and $d_i$, for any $i\in\N$, as follows:
\[
\xymatrix{
\dots\ar[r]& S_{i}\ar[d]^{c_{i}}\ar[rr]^{\sigma_{ii+1}} && S_{i+1}\ar[rr]^{\sigma_{i+1i+2}}\ar[rr]\ar[d]^{c_{i+1}} &&\dots \\
\dots\ar[r] &T_{i}\ar[urr]_{d_i}\ar[rr]_{\tau_{ii+1}} && T_{i+1}\ar[rr]_{\tau_{i+1i+2}}\ar[urr]_{d_{i+1}}\ar[rr]&&\dots
} 
\]
and assume there are strictly increasing sequences of natural numbers $(n_i)_i$ and $(m_i)_i$ such that:

(i) For any $i\leq j$, we have $\sigma_{ij}(M_{n_j})\subseteq M_{n_j}$ and $\tau_{ij}(M_{m_j})\subseteq M_{m_j}$, where $\sigma_{ij}:=\sigma_{j-1j}\circ ...\circ\sigma_{ii+1}$ and $\tau_{ij}:=\tau_{j-1j}\circ ...\circ\tau_{ii+1}$.

(ii) $d_{i}\circ c_i\underset{M_{n_i}}\approx\sigma_{ii+1}$ and $c_{i+1}\circ d_i\underset{M_{m_i}}\approx\tau_{ii+1}$, for any $i\in\N$.

(iii) For any $i\in\N$, we also have $c_i(M_{n_i})\subseteq M_{m_i}$ and $d_i(M_{m_i})\subseteq M_{n_{i+1}}$.

Then $S\simeq T$ as $\Cu$-semigroups.
\end{thm}
In a second part, we focus on building uniform bases of concrete Cuntz semigroups, that is, that can be realized as Cuntz semigroups of a $\CatCa$-algebra. More concretely, we construct a uniform basis \emph{of size $q$}, where $q$ is any supernatural number, for any $\Cu$-semigroup of the form $\Lsc(X,\overline{\N}^r)$, where $X$ is a finite graph and $r\in\N$, and for any $\Cu(A)$, where $A$ is a one-dimensional $\NCCW$ complex.

In a third part, we apply all the above to define topologically equivalent semimetrics on $\Hom_{\Cu}(S,T)$, where $S$ is a uniformly based $\Cu$-semigroup. Finally, the next theorem is showing the uniqueness part of the following \emph{conjecture}: the functor $\Cu$ classifies unitary elements of any unital $\AF$-algebra, in the sense that $\Cu$ classifies any $^*$-homomorphism from $\mathcal{C}(\T)$ to any unital $\AF$-algebra.

\begin{thm}
Let $A$ be a unital $\AF$-algebra and let $\varphi_{u},\varphi_{v}:\mathcal{C}(\T)\longrightarrow A$ be unital $^*$-homo\-morphisms. If $\Cu(\varphi_u)=\Cu(\varphi_v)$, then $\varphi_u$ and $\varphi_v$ are approximately unitarily equivalent.
\end{thm}
 
\textbf{Acknowledgments.} 
The author would like to thank L. Robert for a fruitful collaboration in the Spring of 2018 at the University of Louisiana Lafayette, where he drafted the author in the classification of unitary elements of $\AF$-algebras. The author would also like to mention that the approximate intertwining theorem has been greatly inspired by analogous versions for $\CatCa$-algebras and order unit Spaces, thoroughly detailed in \cite{L01} and in \cite{Th94} respectively. 
Finally, the author is greatly indebted to the referee who raised a numerous amount of relevant points that have reshaped the manuscript in a better way.

\break\section{Preliminaries} 
We use $\CatPoM$ to denote the category of positively ordered monoids. \\
\textbf{The category $\Cu$.}
Let $(S,\leq)$ be a positively ordered monoid and let $x,y$ in $S$. For any two elements $x,y$ of $S$, we say that $x$ is \emph{way-below} $y$ (or $x$ is \emph{compactly-contained} in $y$) if for all increasing sequences $(z_n)_n$ of $S$ that have a supremum (in $S$), if $\sup z_n\geq y$ then there exists $k\in\N$ such that $z_k\geq x$. This is an auxiliary relation on $S$ called the \emph{way-below relation} or the \emph{compact-containment relation}. In particular $x\ll y$ implies $x\leq y$ and we say that $x$ is a \emph{compact element} whenever $x\ll x$. 

We say that $S$ is an abstract Cuntz semigroup, or a $\Cu$-semigroup, if it satisfies the following order-theoretic axioms introduced in \cite{CEI08}: 

(O1): Every increasing sequence of elements in $S$ has a supremum. 

(O2): For any $x\in S$, there exists a $\ll$-increasing (or \textquoteleft rapidly increasing\textquoteright) sequence  $(x_n)_{n\in\N}$ in $S$ such that $\sup\limits_{n\in\N} x_n= x$.

(O3): Addition and the compact-containment relation are compatible.

(O4): Addition and suprema of increasing sequences are compatible.

A \emph{$\Cu$-morphism} is a positively ordered monoid morphism that preserves the compact-contain\-ment relation and suprema of increasing sequences. The category of abstract Cuntz semigroups, written $\Cu$, is the subcategory of $\CatPoM$ whose objects are $\Cu$-semigroups and morphisms are $\Cu$-morphisms. 

A $\Cu$-semigroup $S$ is \emph{countably-based} if it contains a countable subset $B\subseteq S$ such that any element is the supremum of a $\ll$-increasing sequence of $B$. Equivalently, for any $s',s\in S$ such that $s'\ll s$, then there exists $b\in B$ such that $s'\ll b\ll s$. We often refer to $B$ as a (countable) \emph{basis} of $S$. As stated earlier, the topology in $S$ (called the Scott topology; see for instance \cite[Definition 3.2]{C20two}), is based upon the compact-containment relation. Thus, the notion of countably-based can be seen as an analogue of separability while the axiom (O1) can be interpreted as completeness. As a matter of fact, a subset $F\subseteq S$ can be \emph{completed} (respectively \emph{$\ll$-completed}) as follows: 
\[
\begin{array}{ll}
\overline{F}^\leq:= \{\sup\limits_m (m_n)_n\mid (m_n)_n\text{ is an increasing sequence of } F\}.\\
\overline{F}^\ll:= F\cup\{\sup\limits_m (m_n)_n\mid (m_n)_n\text{ is a } \ll\text{-increasing sequence of } F\}.
\end{array}
\]
Naturally we have that $F\subseteq\overline{F}^\ll\subseteq \overline{F}^\leq\subseteq S$. Also observe that $B$ is a basis of $S$ if and only if $\overline{B}^\ll= \overline{B}^\leq= S$. In this case, we might say that $B$ is \emph{dense}, or \emph{$\ll$-dense}, in $S$. For instance, the set $S_\ll:=\{s\in S\mid \text{there exists } t\in S \text{with } s\ll t\}$ is a $\CatPoM$ that is always dense in $S$. 
Finally, we mention that any separable $\CatCa$-algebra $A$ induces a countably-based Cuntz semigroup $\Cu(A)$. In this case, $\Cu(A)$ has a largest element $\infty_A:=\sup\limits_{n\in\N} n[s_A]$, where $s_A$ is any strictly positive (or full positive) element of $A$, and we have that $\Cu(A)_\ll=\{s\in \Cu(A) \mid s\ll \infty_A\}$.

\section{Uniformly Based Cuntz semigroups and approximate intertwinings}

\subsection{Uniformly Based Cuntz semigroups} We introduce the notion of \emph{uniform basis} for a $\Cu$-semigroup. We recall that the topology for $\Cu$-semigroups is order-theoretic, in the sense that it is constructed on the way-below relation. For an element $s$ of a $\Cu$-semigroup $S$, we denote $s_\ll:=\{s'\in S \mid s'\ll s\}$.  Observe that axiom (O2) allows us to approach $s$ from \textquoteleft cut-downs\textquoteright, that is, from elements in $s_\ll$. We mean to construct a setting in which we are able to control these approximations in a uniform way. 

A \emph{uniform basis} of a $\Cu$-semigroup consists of a countable family $(M_n,\epsilon_n)_n$, where $M_n$ is a $\CatPoM$ and $\epsilon_n: S_{\ll}\longrightarrow M_n$ is a weaker version of a $\CatPoM$-morphism that we precise next. The main idea is to be able to uniformly measure how close a \textquoteleft cut-down\textquoteright\ $\epsilon_n(s)$ is from its original element $s$. From this, we will be able in the sequel, to measure how close two $\Cu$-morphisms are from one another and develop a theory of approximate intertwinings for the category $\Cu$.

\begin{dfn}
Let $M,N$ be monoids in $\CatPoM$ and let $\epsilon:M\longrightarrow N$ be a map. We say that $\epsilon$ is a \emph{super-additive} morphism if, for any $g,h\in M$ we have $\epsilon(g)+\epsilon(h)\leq \epsilon(g+h)$.
\end{dfn}

\begin{dfn}
\label{dfn:unifbasedCu} 
We say that a $\Cu$-semigroup $S$ is \emph{uniformly based} if, there exists a countable family $(M_n,\epsilon_n)_{n\in\N}$ where $\left\{\begin{array}{ll} M_n \text{ is a $\CatPoM$ contained in } S_\ll\\ \epsilon_n: S_{\ll}\longrightarrow M_n\ \text{ is an order-preserving super-additive morphism}\end{array}\right.$\, satisfying the following axioms:

(U1): $(M_n)_n$ is a $\subseteq$-increasing sequence of subsets of $S_{\ll}$. (With $M_{-2},M_{-1},M_0:=\{0_S\}$ as convention.)

(U2): For any $n\in\N$, the restriction ${\epsilon_n}_{\mid{\underset{l<n-1}{\bigcup}M_l}}:{\underset{l<n-1}{\bigcup}M_l}\longrightarrow M_n$ is a $\CatPoM$-morphism. 

(U3): For any $s\in S_{\ll}$, $(\epsilon_n(s))_{n\in\N}$ is an increasing sequence in $s_{\ll}$ whose supremum is $s$. 

(U4): The sequence $(\epsilon_n(s))_{n>l}$ is $\ll$-increasing whenever $s\in M_l$.

We refer to $(M_n,\epsilon_n)_n$ as a \emph{uniform basis of $S$}. 
\end{dfn}

\begin{prop}
Let $S$ be a uniformly based $\Cu$-semigroup with uniform basis $(M_n,\epsilon_n)_n$. 
Then $\bigcup\limits_{l}M_l$ is dense in $S$. In particular, $S$ is countably-based whenever $M_n$ is countable for any $n\in\N$. 
\end{prop}

\begin{proof}
 Let $s\in S_\ll$ and let $n\in\N$. Then $\epsilon_{n+1}(\epsilon_n(s))\ll \epsilon_{n+2}(\epsilon_n(s))\leq\epsilon_{n+2}(\epsilon_{n+1}(s))$. We deduce that $(\epsilon_{n+1}(\epsilon_n(s)))_n$ is $\ll$-increasing, and by a \textquoteleft diagonal-type\textquoteright\ argument as in the proof of \cite[Theorem 2.60]{T19}, it can be shown that its supremum is $s$. Again by a diagonal argument, we know that for any $s\in S$, there exists a $\ll$-increasing sequence in $\bigcup\limits_{l}M_l$ whose supremum is s.
\end{proof}

Let us now give a couple of examples of uniformly based $\Cu$-semigroups.

$\bullet$ Let $\overline{\N}:=\N\sqcup\{\infty\}$. The countable family $(\N,{\id_\N})_n$ defines uniform basis of $\overline{\N}$.

$\bullet$ Let $S:=\Cu(M_{2^\infty})$ be the Cuntz semigroup of the CAR-algebra $M_{2^\infty}:=\lim\limits_{\longrightarrow n}(\underset{1}{\overset{n}{\otimes}} M_{2},\id\otimes 1)$. Recall that  $S\simeq \N[\frac{1}{2}]\,\sqcup\,]0,\infty]$, where the mixed sum and mixed order are defined as follows: for any $x_c:=k/2^l \in \N[\frac{1}{2}]$, any $x_s:=k/2^l\in ]0,\infty]$ and any $\epsilon>0$, we have that $x_s\leq x_c\ll x_c \leq x_s+\epsilon$. Moreover, $x_c+x_s=2x_s$. (For more details, see e.g. \cite[Example 4.3.3]{T19}.) 

Now construct for any $n\in\N$ 
\[
\begin{array}{ll} 
\left\{
\begin{array}{ll} 
M_n:=\{\frac{k}{2^l}\in \N[\frac{1}{2}]\mid k\in\N, l\leq n\}.\\
\epsilon_n: \N[\frac{1}{2}]\sqcup\R_{++}\longrightarrow M_n
 \end{array}
 \right.\\
  \hspace{2,68cm}s\longmapsto\max\limits_{x\in M_n}\{ x\ll s\}
   \end{array}
\]
Let $n\in\N$. It is immediate to check that $M_n$ is a well-defined $\CatPoM$ contained in $\N[\frac{1}{2}]\sqcup\R_{++}$ and that $\epsilon_n$ is an order-preserving super-additive morphism. Furthermore, it is easily shown that the countable family $(M_n,\epsilon_n)_n$ satisfies (U1) and (U3) of \autoref{dfn:unifbasedCu}. Finally, for any $l<n-1$ and any $x_c\in M_l$, we have that $\epsilon_n(x_c)=x_c$, since $x_c$ is a compact element that belongs to $M_n$. Thus (U4) is satisfied and the restriction ${\epsilon_n}_{|{\underset{l<n-1}{\bigcup}M_l}}$ is in fact the canonical injection $\underset{l<n-1}{\bigcup}M_l\lhook\joinrel\longrightarrow M_n$, which is a $\CatPoM$-morphism. Thus $S$ admits a uniform basis.\\

Let us extend the latter examples and show that the Cuntz semigroups of separable $\AF$-algebras have a uniform basis. These objects have been studied and characterized in \cite[\S 5.5]{APT14} and we now recall some definitions and properties. 

A $\Cu$-semigroup $S$ is called \emph{simplicial} if $S\simeq \overline{\N}^r$ for some $r\in\N$. Therefore, simplicial $\Cu$-semigroups characterize Cuntz semigroups of finite dimensional $\CatCa$-algebras. As a result, inductive sequences of simplicial Cuntz semigroups characterize Cuntz semigroups of separable $\AF$-algebras. 

\begin{prop}
\label{prop:suma}
Let $S,T$ be uniformly based $\Cu$-semigroups. Then $S\oplus T$ is uniformly based.
\end{prop}

\begin{proof}
The family obtained from the concatenation of a uniform basis of $S$ with a uniform basis of $T$ is a uniform basis of $S\oplus T$.
\end{proof}

\begin{cor}
 Any simplicial $\Cu$-semigroup is uniformly based.
\end{cor}

 We now recall a characterization of inductive limits (of sequences) in $\Cu$, that will be useful in proving that the Cuntz semigroup of any (separable) $\AF$ algebra is uniformly based and in the further course of the manuscript.

\begin{prop}\emph{(See e.g. \cite[Section 2.1]{R12} - \cite[Theorem 2]{CEI08})}
\label{prop:caralimicu}
Consider an inductive sequence $(S_i,\sigma_{ij})_{i\in \N}$ in $\Cu$. Then $(S,\sigma_{i\infty})_{i\in \N}$ is the inductive limit of the inductive sequence if and only if it satisfies the two following properties: 

\emph{(L1)}: For any $s\in S$, there exists $(s_i)_{i\in \N}$ such that $s_i\in S_i$ and $\sigma_{i(i+1)}(s_i)\ll s_{i+1}$, for any $i\in \N$, and such that $s=\sup\limits_{i\in \N}\sigma_{i\infty}(s_i)$. 

\emph{(L2)}: Let $s,t$ be elements in $S_i$ such that $\sigma_{i\infty}(s)\leq\sigma_{i\infty}(t)$. For any $s'\ll s$, there exists $j\geq i$ such that $\sigma_{ij}(s')\ll\sigma_{ij}(t)$.
\end{prop}

\begin{thm}
Let $(S_i,\sigma_{ij})_{i\in \N}$ be an inductive sequence of simplicial $\Cu$-semigroups and let $(S,\sigma_{i\infty})_{i\in \N}$ be its inductive limit. Then $S$ is uniformly based.
\end{thm}

\begin{proof}
We know that each $S_i$ is isomorphic to $\overline{\N}^{r_i}$ for some $r_i\in\N$ and hence $(S_i)_{\ll}\simeq \N^{r_i}$ is a complete lattice for any $i\in\N$. Let $n\in\N$ and define $M_n:=\sigma_{n\infty}((S_n)_\ll)$. Since $\sigma_{n\infty}$ preserves the order, we have that $M_n$ is also a complete lattice. Therefore, we can define $\epsilon_n:S_\ll\longrightarrow M_n$ that maps $s\longmapsto\sup\limits_{x\in M_n}\{x\ll s\}$. Mimicking the UHF case, it is routine to check that $(M_n,\epsilon_n)_{n\in\N}$ satisfies (U1)-(U2)-(U4) and (U3) is almost immediate from (L1) and left to the reader.
\end{proof}

\begin{cor}
Let $A$ be a separable $\AF$-algebra. Then $\Cu(A)$ is uniformly based. 

In particular, let $M_q$ be the $\UHF$-algebra associated to the supernatural number $q:=\prod\limits_{n=0}^\infty p_n$. Then $\Cu(M_q)\simeq(\bigcup\limits_{n\in\N} \frac{1}{p_1\dots p_n}\N)\,\sqcup\,]0,\infty]$ is uniformly based with uniform basis \[
\begin{array}{ll} 
\left\{
\begin{array}{ll} 
M_n:=\frac{1}{p_1\dots p_n}\N.\\
\epsilon_n: (\bigcup\limits_{n\in\N} \frac{1}{p_1\dots p_n}\N)\sqcup\R_{++}\longrightarrow M_n
 \end{array}
 \right.\\
  \hspace{4,14cm}s\longmapsto\max\limits_{x\in M_n}\{ x\ll s\}
   \end{array}
\]
\end{cor}

\subsection{Comparison of \texorpdfstring{$\Cu$}{Cu}-morphisms - Approximate Intertwinings}
In this part, we define a comparison between two abstract $\Cu$-morphisms $\alpha,\beta:S\longrightarrow T$ on a given set $\Lambda\subseteq S$. This allows us to get an analogous notion of one/two-sided approximate intertwinings between two inductive sequences of uniformly based $\Cu$-semigroups.

\begin{dfn}
Let $S,T\in \Cu$, let $\alpha,\beta:S\longrightarrow T$ be two $\Cu$-morphisms and let $\Lambda$ be a subset of $S$.

(i) We say that $\alpha$ and $\beta$ \emph{compare on $\Lambda$} and we write $\alpha\underset{\Lambda}{\simeq}\beta $ if, for any $g',g\in\Lambda$ such that $g'\ll g$ (in $S$), we have that $\alpha(g')\leq \beta(g)$ and $\beta(g')\leq \alpha(g)$.

(ii) We say that $\alpha$ and $\beta$ \emph{strictly compare on $\Lambda$} and we write $\alpha\underset{\Lambda}{\approx}\beta $ if, for any $g',g\in\Lambda$ such that $g'\ll g$ (in $S$), we have that $\alpha(g')\ll \beta(g)$ and $\beta(g')\ll \alpha(g)$.
\end{dfn}

It is immediate that strict comparison on a given set $\Lambda$ implies comparison (on $\Lambda$). Also, two $\Cu$-morphisms $\alpha,\beta:S\longrightarrow T$ compare on $S$ if and only if they strictly compare on $S$ if and only if $\alpha=\beta$. 
For countably-based $\Cu$-semigroups, we can weaken the assumption and the morphisms only have to compare on any (countable) basis of $S$ to be equal. However, to check whether two $\Cu$-morphisms compare on a countable set can still be a hard thing to do. In practice, we often have access to what will call later a \emph{finite} uniform basis, that will make things easier. (We will only have to compare morphisms on finite sets.)

\begin{dfn}
\label{dfn:Cuonesidedapprox}
Let $(S_i,\sigma_{ij})_{i\in\N}$ and $(T_i,\tau_{ij})_{i\in\N}$ be two inductive sequences of uniformly based $\Cu$-semigroups. Let $(S,\sigma_{i\infty})_i$ and $(T,\tau_{i\infty})_i$ be their respective inductive limits in $\Cu$. 
For any $i\in\N$, consider a uniform basis of $S_i$, respectively of $T_i$, that we both denote by $(M_n,\epsilon_n)_n$. (Without referring to the index $i$ to ease notations and the basis we refer to is always clear.)

 Suppose that we have the following diagram:
\[
\xymatrix{
\dots\ar[r]& S_{i}\ar[d]^{c_{i}}\ar[r]^{\sigma_{ii+1}} & S_{i+1}\ar[r]\ar[d]^{c_{i+1}} &\dots \\
\dots\ar[r] &T_{i}\ar[r]_{\tau_{ii+1}} & T_{i+1}\ar[r]&\dots
} 
\]
where $c_i$ is a $\Cu$-morphism for each $i\in\N$. Moreover, assume that there is a strictly increasing sequence of natural numbers $(n_i)_i$ such that:

(i) For any $i\leq j$, we have $\sigma_{ij}(M_{n_j})\subseteq M_{n_j}$, where $\sigma_{ij}:=\sigma_{j-1j}\circ ...\circ\sigma_{ii+1}$.

(ii) $c_{i+1}\circ\sigma_{ii+1}\underset{M_{n_i}}\approx\tau_{ii+1}\circ c_i$, for any $i\in\N$.\\
Then we say the diagram is a \emph{one-sided approximate intertwining}.
\end{dfn}

\begin{lma}
\label{lma:Cuonesidedapprox}
In the context of \autoref{dfn:Cuonesidedapprox}, assume that there is a one-sided approximate intertwining from $(S_i,\sigma_{ij})_{i\in\N}$ to $(T_i,\tau_{ij})_{i\in\N}$. Then for any $i\in\N$, there exists a $\CatPoM$-morphism as follows:
\[
\begin{array}{ll}
\gamma_i:\bigcup\limits_{l}M_l\subseteq (S_i)_\ll\longrightarrow T\\
\hspace{2,65cm}s\longmapsto\sup\limits_{j>i,s_l}(\tau_{j+1\infty}\circ c_{j+1}\circ\sigma_{ij+1}(\epsilon_{n_j}(s)))
\end{array}
\]
where $s_l:=\min\{l\in\N\mid s\in M_l\}$.
\end{lma}

\begin{proof}
Fix $i\in\N$ and let $s\in \bigcup\limits_{l}M_l\subseteq (S_i)_{\ll}$. For any $j>i$, we define
\vspace{-0,3cm}\[
\alpha_j:
\xymatrix{
S_{i}\ar[r]^{\sigma_{ij+1}} & S_{j+1}\ar[d]^{c_{j+1}} && \\
&T_{j+1}\ar[rr]_{\tau_{j+1\infty}} &&T
} 
\]
We claim that $(\alpha_j(\epsilon_{n_j}(s)))_{j>i,s_l}$ is a $\ll$-increasing sequence in $T$. Let $j>i,s_l$. Since $(n_i)_i$ is strictly increasing, we know by (U4) that $(\epsilon_{n_j}(s))_{j>s_l}$ is $\ll$-increasing in $S_i$ towards $s$. We deduce that $\sigma_{ij+1}(\epsilon_{n_j}(s))\ll \sigma_{ij+1}(\epsilon_{n_{j+1}})$ in $S_{j+1}$. 
Moreover, using (i) of \autoref{dfn:Cuonesidedapprox}, we know that $\sigma_{ij+1}(\epsilon_{n_j}(s)),\sigma_{ij+1}(\epsilon_{n_{j+1}}(s))\in M_{n_{j+1}}$. Now using (ii) of \autoref{dfn:Cuonesidedapprox}, we get that $\tau_{j+1j+2}\circ c_{j+1}\circ\sigma_{ij+1}(\epsilon_{n_j}(s))\ll c_{j+2}\circ\sigma_{j+1j+2}\circ\sigma_{ij+1}(\epsilon_{n_{j+1}}(s))$ and hence 
\[\tau_{j+1\infty}\circ c_{j+1}\circ\sigma_{ij+1}(\epsilon_{n_j}(s))\ll\tau_{j+2\infty}\circ c_{j+2}\circ\sigma_{ij+2}(\epsilon_{n_{j+1}}(s)).\] That is, $\alpha_j(\epsilon_{n_j}(s))\ll \alpha_{j+1}(\epsilon_{n_{j+1}}(s))$. Thus, $(\alpha_j(\epsilon_{n_j}(s)))_{j>i,s_l}$ is a $\ll$-increasing sequence in $T$ and its supremum exists.

Lastly, we have to check that $\gamma_i$ is indeed a $\CatPoM$-morphism. It is trivial to see that $\gamma_i$ preserves the order. Now let $s',s\in \bigcup\limits_{l}M_l$. We can find $k\in\N$ big enough such that both $s,s'$ belong to $M_k$. Using (U2), we know that $\epsilon_j(s')+\epsilon_j(s)=\epsilon_j(s'+s)$ for any $j>k+1$. Thus for any $j$ big enough (in particular $j>k+1$), we have that
\[
\tau_{j+1\infty}\circ c_{j+1}\circ\sigma_{ij+1}(\epsilon_{n_j}(s'))+\tau_{j+1\infty}\circ c_{j+1}\circ\sigma_{ij+1}(\epsilon_{n_j}(s))=\tau_{j+1\infty}\circ c_{j+1}\circ\sigma_{ij+1}(\epsilon_{n_j}(s'+s)).
\]
from which the result follows.
\end{proof}

\begin{lma}
\label{lma:PoMmorphCu}
Let $S,T$ be $\Cu$-semigroups and let $B\subseteq S$ be a $\CatPoM$ such that $B$ is dense in $S$. For any $\CatPoM$-morphism $\alpha:B\longrightarrow T$, there exists a naturally associated generalized $\Cu$-morphism $\overline{\alpha}:S\longrightarrow T$ (that is, a $\CatPoM$-morphism that respects suprema of increasing sequences) such that $\overline{\alpha}\leq\alpha$. Moreover, if $\alpha$ preserves the compact-containment relation, then so does $\overline{\alpha}$, or equivalently, $\overline{\alpha}$ is a $\Cu$-morphism.

\end{lma}

\begin{proof}
Let $\alpha:B\longrightarrow T $ be a $\CatPoM$-morphism. We are going to extend $\alpha$ to $S$: Let $(b_n)_n,(c_n)_n$ be two $\ll$-increasing sequences in $B$ such that they have the same supremum in $S$. That is, $\sup\limits_n b_n= \sup\limits_n c_n$ in $S$. First, observe that $\alpha(b_n)_n, \alpha(c_n)_n$ are increasing sequences in $T$ and hence they have a supremum. Also, since $(b_n)_n$ is a $\ll$-increasing sequence, it follows that for any $n\in\N$, there exists some $m\geq n$ such that $b_n\leq c_m$. Besides $\alpha$ is a $\CatPoM$-morphism, so $\alpha(b_n)\leq \alpha(c_m)\leq \sup\limits_n \alpha(c_n)$, for all $n\in\N$. It follows that $\sup\limits_n \alpha(b_n) \leq \sup\limits_n \alpha(c_n)$. By symmetry, we get the converse inequality to conclude the following: $\sup\limits_n \alpha(b_n) =\sup\limits_n \alpha(c_n)$ for any two $\ll$-increasing sequences of $B$ that have the same supremum in $S$. 

Now let $s\in S$. Since $B$ is dense in $S$, then there exists a $\ll$-increasing sequence $(b_n)_n$ in $B$ whose supremum is $s$. We have just proved that $\sup\limits_n\alpha(b_n)$ does not depend on the sequence $(b_n)_n$ chosen. Thus we define
\[
\begin{array}{ll}
\overline{\alpha}:S\longrightarrow T\\
\hspace{0,57cm} s\longmapsto \sup\limits_n\alpha(b_n)
\end{array}
\]
The fact that $\overline{\alpha}\leq \alpha$ is trivial and left to the reader to check. Let us prove that $\overline{\alpha}$ is a generalized $\Cu$-morphism. Let $s,t\in S$ and consider two $\ll$-increasing sequences $(b_n)_n,(c_n)_n$ in $B$ such that $s=\sup\limits_nb_n$ and $t=\sup\limits_nc_n $.

Suppose that $s\leq t$. We have $\sup\limits_n b_n\leq \sup\limits_n c_n $. From what we have proved, this implies that $\sup\limits_n\alpha(b_n)\leq \sup\limits_n\alpha(c_n) $. That is, $\overline{\alpha}(s)\leq \overline{\alpha}(t)$. 
Now put $x:=s+t\in S$. Then $(b_n+c_n)_n$ is a sequence of $B$, since $B$ is a $\CatPoM$, and obviously, it is $\ll$-increasing towards $x$. Using axiom (O4) and the fact that $\alpha$ is a $\CatPoM$-morphism, we see that $\overline{\alpha}(x)=\sup\limits_n\alpha(b_n+c_n)= \sup\limits_n(\alpha(b_n))+ \sup\limits_n(\alpha(c_n))=\overline{\alpha}(s)+\overline{\alpha}(t)$, which proves that $\overline{\alpha}:S\longrightarrow T$ is a well-defined $\CatPoM$-morphism.

Let $(s_n)_n$ be an increasing sequence in $S$ and let $s:=\sup\limits_n s_n$.  Then $(\overline{\alpha}(s_n))_n $ is an increasing sequence in $T$ and hence, it has a supremum. We have to show that $\sup\limits_n (\overline{\alpha}(s_n))=\overline{\alpha}(s)$. On the one hand, $\overline{\alpha}(s_n)\leq\overline{\alpha}(s)$ for any $n\in\N$, from which we obtain $\sup\limits_n(\overline{\alpha}(s_n))\leq \overline{\alpha}(s)$. On the other hand, there exists a $\ll$-increasing sequence $(b_{k,n} )_k$ in $B$ such that $s_n=\sup\limits_n b_{k,n}$, for any $n\in\N$. Since $B$ is countable basis of $S$, it is not hard to construct these sequences recursively in such a way that $b_{k,n}\ll b_{k,n+1}$ for any $k,n\in\N$. Therefore, we can use a \textquoteleft diagonal-type\textquoteright\ argument as in the proof of \cite[Theorem 2.60]{T19} to see that the sequence $(b_{n,n})_n$ is $\ll$-increasing towards $s$. 
 Observe that $\alpha(b_{n,n})\leq \overline{\alpha}(s_n)$ for any $n\in\N$. Passing to supremum on the right side first, and then on the left side, we get that $\sup\limits_n \alpha(b_{n,n})\leq \sup\limits_n (\overline{\alpha}(s_n))$. In other words, $\overline{\alpha}(s)\leq \sup\limits_n (\overline{\alpha}(s_n))$. We conclude that $\alpha:S\longrightarrow T$ is a well-defined generalized $\Cu$-morphism dominated by $\alpha$. 

Finally, assume that $\alpha:B\longrightarrow T$ preserves the compact-containment relation. Then for any two $s,t\in S$ such that $s\ll t$, consider any $\ll$-increasing sequence $(c_n)_n$ in $B$ whose supremum is $t$. Then we can find some $m\in\N$ such that $s\ll c_m\ll c_{m+1}\ll t$. We obtain $\overline{\alpha}(s)\leq\alpha(s)\ll \alpha(c_m)\leq \overline{\alpha}(t)$, which gives us that $\alpha:S\longrightarrow T$ is a $\Cu$-morphism.
\end{proof}

\begin{thm}
\label{cor:gammaconstruction}
Assume that there is a one-sided approximate intertwining from $(S_i,\sigma_{ij})_{i\in\N}$ to $(T_i,\tau_{ij})_{i\in\N}$, where $(S_i,\sigma_{ij})_{i\in\N}$ and $(T_i,\tau_{ij})_{i\in\N}$ are inductive sequences of uniformly based $\Cu$-semigroups with respective limits $(S,\sigma_{i\infty})_i$ and $(T,\tau_{i\infty})_i$ in $\Cu$. 

Then there exists a generalized $\Cu$-morphism $\gamma:S\longrightarrow T$ such that the following diagram is commutative for any $i\in\N$:
\[
\xymatrix{
S_i\ar[rr]^{\sigma_{i\infty}}\ar[dr]_{\gamma_i}&& S\ar[dl]^{\gamma}\\
&T&
}
\]
where $\gamma_i$ is obtained from \autoref{lma:Cuonesidedapprox} combined with \autoref{lma:PoMmorphCu}.
\end{thm}

\begin{proof}
Combining \autoref{lma:Cuonesidedapprox} together with \autoref{lma:PoMmorphCu}, we have in fact constructed generalized $\Cu$-morphisms $\gamma_i:S_i\longrightarrow T$, for any $i\in\N$, such that $\gamma_i=\gamma_j\circ\sigma_{ij}$ for any $i\leq j$. By universal properties of inductive limits, we obtain a (unique) generalized $\Cu$-morphism $\gamma:S\longrightarrow T$ that commutes with all $\gamma_i$. More explicitly, one can build $\gamma$ as follows: 

Consider $\gamma:\bigcup\limits_{i\in\N}\sigma_{i\infty}(S_i)\longrightarrow T$ that sends $s_i\longmapsto\gamma_i(s_i)$. It is easy to check that $\gamma$ is a well-defined $\CatPoM$-morphism. Moreover, (L1) says that $\bigcup\limits_{i\in\N}\sigma_{i\infty}(S_i)$ is dense in $S$. Thus, using \autoref{lma:PoMmorphCu}, the following map is a well-defined generalized $\Cu$-morphism:
\[
	\begin{array}{ll}
		\gamma: S\longrightarrow T\\
	\hspace{0,6cm} s\longmapsto \sup\limits_i(\gamma_i(s_i))
	\end{array}
	\]
where $(s_i)_{i\in\N}$ is any sequence as in (L1). The universal property of commutativity is easy to check and left to the reader.
\end{proof}

\begin{dfn}
\label{dfn:2sidedapprox}
Let $(S_i,\sigma_{ij})_{i\in\N}$ and $(T_i,\tau_{ij})_{i\in\N}$ be two inductive sequences of uniformly based $\Cu$-semigroups. Let $(S,\sigma_{i\infty})_i$ and $(T,\tau_{i\infty})_i$ be their respective inductive limits. 
For any $i\in\N$, consider a uniform basis of $S_i$, respectively of $T_i$, that we both denote by $(M_n,\epsilon_n)_n$. (Without referring to the index $i$ to ease notations and the basis we refer to is always clear.)

 Suppose that we have the following diagram:
\[
\xymatrix{
\dots\ar[r]& S_{i}\ar[d]^{c_{i}}\ar[rr]^{\sigma_{ii+1}} && S_{i+1}\ar[rr]^{\sigma_{i+1i+2}}\ar[rr]\ar[d]^{c_{i+1}} &&\dots \\
\dots\ar[r] &T_{i}\ar[urr]_{d_i}\ar[rr]_{\tau_{ii+1}} && T_{i+1}\ar[rr]_{\tau_{i+1i+2}}\ar[urr]_{d_{i+1}}\ar[rr]&&\dots
} 
\]
where $c_i$ and $d_i$ are $\Cu$-morphisms, for any $i\in\N$. Moreover, assume that there are two strictly increasing sequences of natural numbers $(n_i)_i$ and $(m_i)_i$ such that:

(i) For any $i\leq j$, we have $\sigma_{ij}(M_{n_j})\subseteq M_{n_j}$ and $\tau_{ij}(M_{m_j})\subseteq M_{m_j}$, where $\sigma_{ij}:=\sigma_{j-1j}\circ ...\circ\sigma_{ii+1}$ and $\tau_{ij}:=\tau_{j-1j}\circ ...\circ\tau_{ii+1}$.

(ii) $d_{i}\circ c_i\underset{M_{n_i}}\approx\sigma_{ii+1}$ and $c_{i+1}\circ d_i\underset{M_{m_i}}\approx\tau_{ii+1}$, for any $i\in\N$.

(iii) For any $i\in\N$, we also have $c_i(M_{n_i})\subseteq M_{m_i}$ and $d_i(M_{m_i})\subseteq M_{n_{i+1}}$.\\
Then we say the diagram is a \emph{two-sided approximate intertwining}.
\end{dfn}

\begin{lma}
\label{lma:PoMisoCu}\emph{(\cite[Lemma 4.1.10]{C20one})}
Let $S$ be a $\Cu$-semigroup and let $T$ be a $\CatPoM$. Let $f:S\longrightarrow T$ be a $\CatPoM$-isomorphism. Then, $T$ is a $\Cu$-semigroup and $f$ is a $\Cu$-isomorphism. A fortiori, $S\simeq T$ as $\Cu$-semigroups.
\end{lma}

\begin{thm}
Assume that there is a two-sided approximate intertwining from $(S_i,\sigma_{ij})_{i\in\N}$ to $(T_i,\tau_{ij})_{i\in\N}$, where $(S_i,\sigma_{ij})_{i\in\N}$ and $(T_i,\tau_{ij})_{i\in\N}$ are inductive sequences of uniformly based $\Cu$-semigroups with respective limits $(S,\sigma_{i\infty})_i$ and $(T,\tau_{i\infty})_i$ in $\Cu$. 

Then $S\simeq T$ as $\Cu$-semigroups.
\end{thm}

\begin{proof}
Consider the two generalized $\Cu$-morphisms $\gamma:S\longrightarrow T$ and $\delta:T\longrightarrow S$ explicitly constructed in the proof of \autoref{cor:gammaconstruction}. We are going to prove that $\gamma$ and $\delta$ are inverses of one another and the result will follow.
It is enough to show that for any $i\in\N$ and any $s\in S_i$, we have $\delta\circ\gamma_i(s)=\sigma_{i\infty}(s)$. 
As a matter of fact, since $(S_i)_\ll$ is dense in $S_i$ and since $\sigma_{i\infty}, \delta,\gamma_i$ are morphisms that preserve suprema of increasing sequences, it is enough to show the said property for every $s\in(S_i)_\ll$. 
Fix $i\in\N$ and let $s\in (S_i)_\ll$.  To begin, we observe the following:
\begin{align*}
	\delta	\circ\gamma_i(s)&=\delta\,(\,\sup\limits_{j}\tau_{j+1\infty}\circ c_{j+1}\circ\sigma_{ij+1}(\epsilon_{n_j}(s))\,)\\
	&=\sup\limits_{j} \,\delta\circ\tau_{j+1\infty}\circ c_{j+1}\circ\sigma_{ij+1}(\epsilon_{n_j}(s))\\
	&=\sup\limits_{j}\,\delta_{j+1}(\,c_{j+1}\circ \sigma_{ij+1}(\epsilon_{n_j}(s))\,)\\
&=\sup\limits_{j}\,\sup\limits_{j'>j+1}\,\sigma_{j'+1\infty}\circ d_{j'}\circ\tau_{j+1j'+1}(\,\epsilon_{n_{j'}}(c_{j+1}\circ\sigma_{ij+1}(\epsilon_{n_j}(s)))\,).
\end{align*}
Let us prove that $\delta\circ\gamma_i(s)\leq\sigma_{i\infty}(s)$ and the opposite inequality will be shown similarly: 

Since $(\epsilon_{n_j}(s))_{j>s_l}$ is $\ll$-increasing, for any $j>i,s_l$ we have that
\[
\begin{array}{ll}
\sigma_{ij+1}(\epsilon_{n_j}(s))\ll \sigma_{ij+1}(\epsilon_{n_{j+1}}(s)) \text{ in } S_{j+1}.\\
c_{j+1}\circ\sigma_{ij+1}(\epsilon_{n_{j-1}}(s))\ll c_{j+1}\circ\sigma_{ij+1}(\epsilon_{n_j}(s)) \text{ in } T_{j+1}.
\end{array}
\]
Using (i) and (iii) of \autoref{dfn:2sidedapprox}, we know that the above elements belong to $M_{n_{j+1}}\subseteq S_{j+1}$ and $M_{m_{j+1}}\subseteq  T_{j+1}$ respectively. Thus we can apply (ii) \autoref{dfn:2sidedapprox} as follows:
\[
\begin{array}{ll}
d_{j+1}\circ c_{j+1}\circ\sigma_{ij+1}(\epsilon_{n_j}(s))\ll \sigma_{ij+2}(\epsilon_{n_{j+1}}(s)) \text{ in } S_{j+2}.\\
\tau_{j+1j+2}\circ c_{j+1}\circ\sigma_{ij+1}(\epsilon_{n_{j-1}}(s))\ll c_{j+2}\circ d_{j+1}\circ c_{j+1}\circ\sigma_{ij+1}(\epsilon_{n_j}(s)) \text{ in } T_{j+2}.
\end{array}
\]
Using (i) and (iii) again, we know that the above elements belong to $M_{n_{j+2}}\subseteq S_{j+2}$ and $M_{m_{j+2}}\subseteq T_{j+2}$ respectively. Thus we can apply (ii) again, and repeating this process, we obtain that for any $j'>j+1$
\[
\begin{array}{ll}
d_{j'+1}\circ c_{j'+1}\circ \dots\circ d_{j+1}\circ c_{j+1}\circ\sigma_{ij+1}(\epsilon_{n_j}(s))\ll \sigma_{ij'+2}(\epsilon_{n_{j+1}}(s)) \text{ in } S_{j'+2}.\\
\tau_{j+1j'+1}\circ c_{j+1}\circ\sigma_{ij+1}(\epsilon_{n_{j-1}}(s))\ll  c_{j'+1}\circ d_{j'}\circ\dots \circ d_{j+1}\circ c_{j+1}\circ\sigma_{ij+1}(\epsilon_{n_j}(s)) \text{ in } T_{j'+1}.
\end{array}
\]
On the other hand, since $(\epsilon_{n_{j'}}(c_{j+1}\circ\sigma_{ij+1}(\epsilon_{n_{j-1}}(s))))_{j'}$ is $\ll$-increasing towards \break$c_{j+1}\circ\sigma_{ij+1}(\epsilon_{n_{j-1}}(s))$, for any $j'>j$ we have that
\[\tau_{j+1j'+1}(\epsilon_{n_{j'}}(c_{j+1}\circ\sigma_{ij+1}(\epsilon_{n_{j-1}}(s))))\ll \tau_{j+1j'+1}\circ c_{j+1}\circ\sigma_{ij+1}(\epsilon_{n_{j-1}}(s)).\] So we conclude that
\[
d_{j'+1}\circ\tau_{j+1j'+1}(\epsilon_{n_{j'}}(c_{j+1}\circ\sigma_{ij+1}(\epsilon_{n_{j-1}}(s))))\ll \sigma_{ij'+2}(\epsilon_{n_{j+1}}(s)) \text{ in } S_{j'+2}.
\]
Composing with $\sigma_{j'+1\infty}$, we obtain for any $j,j'$ big enough such that $j'>j+1$:
\[
\sigma_{j'+2\infty}\circ d_{j'+1}\circ\tau_{j+1j'+1}(\epsilon_{n_{j'}}(c_{j+1}\circ\sigma_{ij+1}(\epsilon_{n_{j-1}}(s))))\ll \sigma_{i\infty}(\epsilon_{n_{j+1}}(s)) \text{ in } S_{j'+2}\text{ in } S.
\]
Now, taking suprema over $j'$ first and then over $j$, we conclude that $\delta\circ\gamma_i(g)\leq\sigma_{i\infty}(g)$.
The converse inequality $\sigma_{i\infty}(g)\leq \delta\circ\gamma_i(g)$ is shown similarly by using the symmetrical comparison of morphisms. 

We conclude that for any $i\in\N$ and any $s\in S_i$, we have $\delta\circ\gamma_i(s)=\sigma_{i\infty}(s)$. It follows that $\delta\circ\gamma=\id_S$. Symmetrically, we have $\gamma\circ\delta=\id_T$. Now, since any $\CatPoM$-isomorphism between two $\Cu$-semigroups is in fact a $\Cu$-isomorphism (see \autoref{lma:PoMisoCu}), we conclude $S\simeq T$ as $\Cu$-semigroups through $\gamma$ and $\delta=\gamma^{-1}$.

\end{proof}
\section{Examples}
We have already seen that inductive limits of simplicial $\Cu$-semigroups are uniformly based. This part of the manuscript reveals other classes of $\Cu$-semigroups that admit a uniform basis. Namely, we aim to show that the Cuntz semigroup of $\mathcal{C}(X)$, where $X$ is a suitable topological space and the Cuntz semigroup of any one-dimensional $\NCCW$ complex both admit a uniform basis. 
To begin our study, we focus on the specific case of $\Cu(\mathcal{C}_0(]0,1[))$ to then extend our result to concrete Cuntz semigroups that are realized as $\Cu(\mathcal{C}(X)\otimes A)$, where $X$ is a compact one-dimensional CW complex (or equivalently, a finite graph) and $A$ is a finite dimensional $\CatCa$-algebra. As a consequence, the Cuntz semigroups of $\mathcal{C}([0,1])$ and $\mathcal{C}(\T)$ are both uniformly based. Then, we expand our results and show that the Cuntz semigroup of any one-dimensional $\NCCW$ complex is uniformly based.

We start by recalling some facts about lower-semicontinuous functions and how they relate to these concrete Cuntz semigroups. We then introduce a notion of \emph{chain-generating set} in view of simplifying the construction of uniform bases. (Note that all the concrete and abstract Cuntz semigroups that we study in the sequel contain such a set.)\\

\vspace{-0,2cm}$\hspace{-0,34cm}\bullet\,\,\textbf{Lower-semicontinuous functions}.$ Let $X$ be a topological space and let $S$ be a $\Cu$-semigroup. We say that a map $f:X\longrightarrow S$ is \emph{lower-semicontinuous} if for any $s\in S$, the set $\{t\in X \mid s\ll f(t)\}$ is open in $X$. We write $\Lsc(X,S)$ for the set of lower-semicontinuous functions from $X$ to $S$. 
Whenever $X$ is a second-countable locally compact Hausdorff space with covering dimension at most one and $S$ is a countably-based $\Cu$-semigroup, then $\Lsc(X,S)$ is itself a $\Cu$-semigroup. 
Moreover, for any separable $\CatCa$-algebra $A$ of stable rank one such that $\K_1(I)=0$ for every ideal of $A$, then $\Cu(\mathcal{C}_0(X)\otimes A)\simeq \Lsc(X,\Cu(A))$. (See \cite[Theorem 5.15 - Theorem 3.4]{APS11}.) In particular, we have $\Cu(\mathcal{C}_0(X))\simeq \Lsc(X,\overline{\N})$ for any such $X$.\\

\vspace{-0,2cm}$\hspace{-0,34cm}\bullet\,\,\textbf{Chain-decompositions - Chain-generating sets}.$ Let $(P,\leq)$ be a partially ordered set. Any totally ordered subset of $P$ will be called a \emph{chain}. We will be mostly interested in subsets that are totally ordered for the opposite order $\geq$, that  we may refer to as \emph{descending chains}. We now introduce a notion of \emph{chain-generating set} for both a positively ordered monoid and a $\Cu$-semigroup, that roughly consists of a (proper) subset which satisfies a \textquoteleft uniqueness and existence\textquoteright-type condition as follows:

\begin{dfn}
\label{dfn:chainmonoid}
Let $M$ be a $\CatPoM$.  We say that a (proper) subset $\Lambda\subsetneq M$ is a \emph{chain-generating set of $M$} if

(i) For any finite descending chains $(m'_n)_1^{k'}$ and $(m_n)_1^k$ in $\Lambda$, we have $\sum\limits_{n=0}^{k'} m_n \leq \sum\limits_{n=0}^k m_n$ if and only if $m'_n\leq m_n$ for all $0\leq n\leq \min(k',k)$. 

(ii) For any $m\in M$, there exists a finite descending chain $(m_n)_n$ in $\Lambda$ such that $m=\sum\limits_{n=0}^k m_n$. 

Note that we automatically get uniqueness of such a descending chain, up to adding or removing \textquoteleft tailing zeros\textquoteright. Therefore, we refer to it as \emph{the chain-decomposition of $m$}. (We omit the reference to the chain-generating set $\Lambda$ when the context is clear.)
\end{dfn}

\begin{dfn}\label{prg:chaindcp} Let $S$ be a $\Cu$-semigroup. We say that a subset $\Lambda\subseteq S_\ll$ is a \emph{chain-generating set of $S$} if

(i) For any countable descending chains $(s_n)_n$ and $(t_n)_n$ in $\Lambda$, we have $\sum\limits_{n=0}^\infty s_n \leq \sum\limits_{n=0}^\infty t_n$ if and only if $s_n\leq t_n$ for all $n\in\N$. 

(ii) For any $s\in S$, there exists a countable descending chain $(s_n)_n$ in $\Lambda$ such that $s=\sum\limits_{n=0}^\infty s_n$.

Note that we automatically get uniqueness of such a chain, up to adding or removing \textquoteleft tailing zeros\textquoteright. Therefore, we refer to it as \emph{the chain-decomposition of $s$}. (We omit the reference to the chain-generating set $\Lambda$ when the context is clear.)
\end{dfn}

\begin{prop}
Let $X$ be a (second-countable locally compact Hausdorff) topological space of covering dimension at most one. Then $\Lambda:=\Lsc(X,\{0,1\})$ is a chain-generating set of the $\Cu$-semigroup $\Lsc(X,\overline{\N})$. 
\end{prop}

\begin{proof}
We recall that $\Lsc(X,\overline{\N})$ is a $\Cu$-semigroup. Also, observe that for any two open sets $U,V\subseteq X$, $1_U\leq 1_V$ if and only if $U\subseteq V$ and $1_U\ll 1_V$ if and only if $\overline{U}\subseteq V$. 
Let $f\in\Lsc(X,\overline{\N})$ and write $V_n:=f^{-1}(]n;+\infty])$. By the lower-semicontinuity of $f$, $(V_n)_n$ is a well-defined descending chain of open sets of $X$ that naturally satisfies $\overset{\infty}{\underset{n=0}{\sum}}1_{V_n}=f$. Now let $(1_{V_n})_n$ and $(1_{W_n})_n$ be descending chains in $\Lambda$. If $1_{V_n}\leq 1_{W_n}$ for any $n\in \N$, then $\sum\limits_{n=0}^k 1_{V_n}\leq \sum\limits_{n=0}^k 1_{W_n}$ for any $k\in\N$. Passing to suprema over $k$, on the right side first and then on the left side, we deduce that $\sum\limits_{n=0}^\infty 1_{V_n}\leq \sum\limits_{n=0}^\infty 1_{W_n}$. Conversely, if $\sum\limits_{n=0}^\infty 1_{V_n}\leq \sum\limits_{n=0}^\infty 1_{W_n}$, then for any $t\in X$, we have $(\sum\limits_{n=0}^\infty 1_{V_n})(t)\leq (\sum\limits_{n=0}^\infty 1_{W_n})(t)$. Observe that in fact, $(\sum\limits_{n=0}^\infty 1_{V_n})(t)=\card(\{n\mid t\in V_n\})$, hence we deduce that $1_{V_n}\leq 1_{W_n}$ for any $n\in\N$. 
\end{proof}

\begin{thm}
\label{prop:chaindcp}
Let $S$ be a $\Cu$-semigroup that has a chain-generating set $\Lambda$.

(i) $\Lambda$ is an order-hereditary and upward-directed subset of $S_\ll$. 

(ii) The chain-decomposition $(s_n)_n$ of $s$ has finitely many non-zero elements whenever $s\in S_\ll$.

(iii) Let $s,t\in S$ and let $(s_n)_n,(t_n)_n$ be their respective chain-decompositions. If $s\ll t$ then $s_n\ll t_n$ for any $n\in\N$. The converse is true whenever $s\in S_\ll$.
\end{thm}

\begin{proof}
(i) Let $s\in S$ and $\lambda\in \Lambda$ be such that $s\leq \lambda$. By uniqueness of the chain-decomposition, we know that the sequence $(\lambda,0,\dots)_n$ is the chain-decomposition of $\lambda$. Consider $(s_n)_n$ to be the chain-decomposition of $s$. Since $s\leq \lambda$, we deduce that $s_0\leq \lambda$ and that $s_n=0$ for all $n\neq 0$. Thus, $s=s_0$ which proves that $s\in\Lambda$. Further, let $\lambda_1,\lambda_2\in \Lambda$. Consider $s:=\lambda_1+\lambda_2$ and let $(s_n)_n$ be the chain-decomposition of $s$. Then $s_0$ is an element of $\Lambda$ and moreover, we have $\lambda_1,\lambda_2\leq s_0$ since $\lambda_1,\lambda_2\leq s$.

(ii) Let $s\in S_\ll$. There exists $t\in S$ such that $s\ll t$. Let $(s_n)_n,(t_n)_n$ be the respective chain-decompositions of $s,t$. Observe that the sequence $(\sum\limits_{n=0}^k t_n)_k$ is $\leq$-increasing towards $t$. Hence there exists a finite stage $l\in\N$ such that $s\leq \sum\limits_{n=0}^l t_n=: t'$. By uniqueness of the chain-decomposition, we know that $(t_0,\dots,t_l,0,\dots)$ is the chain-decomposition of $t'$ from which we deduce that $s_k=0$ for any $k\geq l+1$. 

(iii) Assume that $s\ll t$.  For any $n\in\N$, consider a $\ll$-increasing sequence $(t_{n,i})_{i\in\N}$ whose supremum is $t_n$, obtained from (O2). Then $(\sum\limits_{n=0}^\infty t_{n,i})_i$ is an increasing sequence whose supremum is $t$. Thus, there exists $l\in\N$ with $s\leq \overset{\infty}{\underset{n=0}\sum} t_{n,l}$. 
We deduce that $s_n\leq t_{n,l}$ for any $n\in\N$, which gives us that $s_n\ll t_n$ for any $n\in\N$.
Conversely, assume that $s\in S_{\ll}$ and that $s_n\ll t_n$ for any $n\in \N$. Since all but finitely many $s_n$ are zero, we conclude that $s\ll t$ using axiom (O3).
\end{proof}

\subsection{The Cuntz semigroup of continuous functions over the open interval}
\label{sec:A}
In this part, we show that  $\Cu(\mathcal{C}_0(]0,1[))\simeq \Lsc(]0,1[,\overline{\N})$ is a uniformly based $\Cu$-semigroup. In order to do so, we explicitly construct uniform basis \textquoteleft from\textquoteright\ the supernatural number $2^\infty$. This will give rise to a handful of uniform bases obtained from any supernatural number  by mimicking a similar process.
 
\begin{prg}\label{prg:cover}Let $n\in\N$ and let $(x_k)_0^{2^n}$ be an equidistant partition of $]0,1[$ of size $1/2^n$. For any $k\in \{1,\dots 2^n\}$, we define the open interval $U_k:=]x_{k-1};x_k[$ of $]0,1[$. It is immediate to see that $\mathcal{U}:=\{\overline{U_k}\}_1^{2^n}$ is a (finite) closed cover of $]0,1[$.
Now define for any $n\geq 1$
\[
M_n:=\{f\in\Lsc(]0,1[,\N)\mid f_{|U_k} \text{ is constant for any }k\in \{1,\dots 2^n\}\}.
\]
\end{prg}
Trivially, $(M_n)_n$ is a $\subseteq$-increasing sequence of positively ordered monoids in $S_\ll$. (We recall that $M_0$ is fixed to $\{0\}$ by convention.) Also, observe that for any $n\in\N$ and any $g\in M_n$, the chain-decomposition of $g$ is a descending chain in $\Lambda\cap M_n$. Thus, $\Lambda$ induces a chain-generating set $\Lambda_n:=\Lambda\cap M_n$ of $M_n$. This will be used several times in what follows. 
We now have to build an order-preserving super-additive morphism $\epsilon_n:S_\ll\longrightarrow M_n$ for any $n\in\N$, satisfying (U2)-(U3)-(U4).
We will use the following lemmas:

\begin{lma}
\label{lma:ccelmnt}
Let $n\in\N$ and let $g',g\in M_n$ such that $g'\ll g$. There exists an element $h\in M_{n+1}$ such that $g'\ll h \ll g$.
\end{lma}

\begin{proof}
Let $(x_k)_0^{2^n},(y_k)_0^{2^{n+1}}$ be equidistant partitions of $]0,1[$ of size $1/2^n,1/2^{n+1}$ respectively. We consider the closed (finite) covers $\mathcal{U}:=\{\overline{U_k}\}_1^{2^n}$ and $\mathcal{W}:=\{\overline{W_k}\}_1^{2^{n+1}}$ of $]0,1[$, constructed as above.

We first assume that both $g',g$ belong to the induced chain-generating set $\Lambda_n:=\Lambda\cap M_n$ of $M_n$ and the general case will follow using the chain-decomposition properties (see \autoref{prg:chaindcp} and \autoref{prop:chaindcp}). There exist open sets $V',V$ of $]0,1[$ such that $g'=1_{V'}, g=1_V$. Furthermore, both $V,V'$ have a finite number of (open) connected components. Again, we first assume that $V,V'$ are connected open sets of $]0,1[$ and we repeat the process finitely many times to obtain the result. Observe that in this case, we have 
\[
V':=U_{l'}\cup(\bigcup\limits_{k=l'+1}^{r'-1}\overline{U_{k}})\cup U_{r'} \hspace{2cm}
V:=U_{l}\cup(\bigcup\limits_{k=l+1}^{r-1}\overline{U_{k}})\cup U_{r}
\]
for some $l < l' \leq r'< r$. 

On the other hand, observe that $W_{2k-1}\cup W_{2k}= U_k\setminus\{y_{2k-1}\}$ for any $1\leq k\leq 2^n$. Let us construct the following open set of $]0,1[$:
\[
W:=W_{2l'}\cup(\bigcup\limits_{k=l'+1}^{r'-1}\overline{U_{k}})\cup W_{2r'-1}
\]
and we define $h:=1_{W}$. By construction, $h\in\M_{n+1}$ and $g'\ll h\ll g$, which ends the proof.
\end{proof}

\begin{lma}
Let $f\in\Lsc(]0,1[,\N)$. For any $n\in\N$, the set $\{g\in M_n\mid g\ll f\}$ has a largest element, that we write $\epsilon_n(f)$. Further, $(\epsilon_n(f))_n$ is an increasing sequence whose supremum is $f$.
\end{lma}

\begin{proof}
Let us explicitly build $\epsilon_n(f)$. Since $\Lambda:=\Lsc(]0,1[,\{0,1\})$ is a chain-generating set of $\Lsc(]0,1[,\N)$, we first prove the result for element $\Lambda$ and the general case will follow using the chain-decomposition properties (see \autoref{prg:chaindcp} and \autoref{prop:chaindcp}). Let $V\subseteq \,]0,1[$ be an open set and write $f:=1_V$. For any $n\in\N$, construct $\epsilon_n(f)$ recursively as follows:

	\vspace{0,2cm}(1) For any $1\leq k\leq 2^n$, $\left\{\begin{array}{ll}\epsilon_n(f)_{|U_k}:=1 \text{ if } \overline{U_k}\subseteq V.\\ \epsilon_n(f)_{|U_k}:=0 \text{ otherwise}. \end{array}\right.$\\
	
	\vspace{0,2cm}(2) For any $x\in \,]0,1[\setminus (\underset{k=1}{\overset{2^n}\cup}U_k)$, $\left\{\begin{array}{ll}\epsilon_n(f)(x):=1 \text{ if } \epsilon_n(f)_{|U_k}\neq 0 \text{ for all }k \text{ such that } x\in \overline{U_k}.\\ \epsilon_n(f)(x):=0 \text{ otherwise}. \end{array}\right.$\\
	
We leave to the reader to check that $\epsilon_n(f)$ is indeed the largest element of $\{g\in M_n\,| \,g\ll f\}$. Moreover, it is trivial to see that $\epsilon_n(f)\leq \epsilon_{n+1}(f)\ll f$, for any $n\in\N$. Finally, by density of $\N[\frac{1}{2}]$ in $\R$, we know that for any $x\in V$, there exists a pair $(k,n)$ of natural numbers, with $n$ big enough such that $x\in[\frac{k}{2^n};\frac{k+1}{2^n}]$ and such that $[\frac{k}{2^n};\frac{k+1}{2^n}]\subseteq V$. We hence deduce that for any $x\in V$, there exists $n\in\N$ big enough such that $\epsilon_n(f)(x)=1$ and we conclude that $\sup \epsilon_n(f)=f$.
\end{proof}

\begin{rmk}Combining the two previous lemmas, it is immediate that for any $g\in M_l$, the sequence $(\epsilon_n(g))_{n>l}$ is $\ll$-increasing (towards $g$) .\end{rmk}

\begin{lma}
\label{lma:ccelmnt2}
Let $n\in\N$. The assignment \vspace{-0,708cm}\[\begin{array}{ll}\hspace{6,45cm}\epsilon_n:\Lsc(]0,1[,\N)\longrightarrow M_n\hspace{0,5cm} \text{ is a well-defined order-}\\ \hspace{8,78cm} f\longmapsto \epsilon_n(f)\end{array}\]
 preserving super-additive morphism. Moreover, the restriction ${\epsilon_n}_{|{\underset{l<n-1}{\bigcup}M_l}}:{\underset{l<n-1}{\bigcup}M_l}\longrightarrow M_n$ is a $\CatPoM$-morphism. 
\end{lma}

\begin{proof}
Let $n\in\N$. For $f\in\Lsc(]0,1[,\N)$, the element $\epsilon_n(f)$ is constructed as the largest element of $\{g\in M_n\,| \,g\ll f\}$ and hence, we automatically get that $\epsilon_n$ is an order-preserving super-additive morphism. 
Now we have to prove that the restriction ${\epsilon_n}_{|{\underset{l<n-1}{\bigcup}M_l}}$ is a $\CatPoM$-morphism, for any $n$. Note $\epsilon_0,\epsilon_1,\epsilon_2$ are restricted to $\{0_S\}$ and hence automatically satisfy (U2). 

Let $n> 2$ and let $l\in\N$ be such that $l<n-1$. Let $f',f\in M_l$. For the rest of the proof, let $(x_k)_0^{2^n},(y_k)_0^{2^l}$ be equidistant partitions of $]0,1[$ of size $1/2^n,1/2^l$ respectively. We consider the finite closed covers $\mathcal{U}:=\{\overline{U_k}\}_1^{2^l}$ and $\mathcal{W}:=\{\overline{W_k}\}_1^{2^n}$ of $]0,1[$, constructed as in \autoref{prg:cover}.

We first assume that both $f,f'$ belong to the induced chain-generating set $\Lambda_l:=\Lambda\cap M_l$ of $M_l$. Hence, there exist open sets $V,V'$ of $]0,1[$ such that $f=1_V,f'=1_{V'}$. Furthermore, both $V,V'$ have a finite number of (open) connected components. Thus we can also suppose, in a first time, that $V,V'$ are connected open sets of $]0,1[$ and we repeat the process finitely many times to obtain the result. 
Observe that $(V\cup V',V\cap V')$ is the chain-decomposition of $f+f'$. By the chain-decomposition properties (see \autoref{prg:chaindcp} and \autoref{prop:chaindcp}), it is easy to check that $\epsilon_n(f+f')=\epsilon_n(1_{V\cup V'})+\epsilon_n(1_{V\cap V'})$. Therefore, if $V\cap V'=\emptyset, V$ or $V'$ (in other words, either $V$ and $V'$ are disjoint sets or one contains the other) then the result is trivial. Else, we have 
\[
V:=U_{l}\cup(\bigcup\limits_{k=l+1}^{r-1}\overline{U_{k}})\cup U_{r} \hspace{2cm}V':=U_{l'}\cup(\bigcup\limits_{k=l'+1}^{r'-1}\overline{U_{k}})\cup U_{r'}
\]
for some $l< l'\leq r<  r'$. (Or $l'< l\leq r' < r'$ and the proof is similar.)

Since $\mathcal{U}$ and $\mathcal{W}$ have the same endpoints, it follows that for any $1\leq k\leq 2^l$ we can find $2^{n-l}$ open sets $W_{k,i}$ of $\mathcal{U}$ such that $W_{k,i}\subseteq U_k$ and such that $\{\overline{W_{k,i}}\}_{i=1}^{2^{n-l}}$ is a finite closed cover of $\overline{U_k}$. Thus, we can explicitly compute
\[
	\left\{\begin{array}{ll}
		\supp(\epsilon_n(f))= (U_{l}\setminus \overline{W_{l,1}})\cup(\bigcup\limits_{k=l+1}^{r-1}\overline{U_{k}})\cup (U_{r}\setminus \overline{W_{r,2^{n-l}}}). \\
		\supp(\epsilon_n(f'))= (U_{l'}\setminus \overline{W_{l',1}})\cup(\bigcup\limits_{k=l'+1}^{r'-1}\overline{U_{k}})\cup (U_{r'}\setminus \overline{W_{r',2^{n-l}}}).		
		\end{array}
	\right.
\]
And also
\[
	\left\{\begin{array}{ll}
		\supp(\epsilon_n(1_{V\cup V'}))= (U_{l}\setminus \overline{W_{l,1}})\cup(\bigcup\limits_{k=l+1}^{r'-1}\overline{U_{k}})\cup (U_{r'}\setminus \overline{W_{r',2^{n-l}}}). \\
		\supp(\epsilon_n(1_{V\cap V'}))= (U_{l'}\setminus \overline{W_{l',1}})\cup(\bigcup\limits_{k=l'+1}^{r-1}\overline{U_{k}})\cup (U_{r}\setminus \overline{W_{r,2^{n-l}}}).		
		\end{array}
	\right.
\] 
On the other hand, we know that $n-l\geq 2$ and hence, each $U_k$ contains at least $4$ open sets $W_{k,i}$. This ensures that $(U_{r}\setminus \overline{W_{r,2^{n-l}}})$ overlaps with $(U_{l'}\setminus \overline{W_{l',1}})$, in the sense that $(U_{r}\setminus \overline{W_{r,2^{n-l}}})\cap(U_{l'}\setminus \overline{W_{l',1}})\neq\emptyset$. A fortiori, $\supp(\epsilon_n(f))\cap\supp(\epsilon_n(f'))\neq \emptyset$. (Even in the particular case where $l'=r$.) As a consequence, we deduce that $\supp(\epsilon_n(f))\cap\supp(\epsilon_n(f'))=\supp(\epsilon_n(1_{V\cap V'}))$ and that $\supp(\epsilon_n(f))\cup\supp(\epsilon_n(f'))=\supp(\epsilon_n(1_{V\cup V'}))$. Finally, we know that
\begin{align*}
\epsilon_n(f)+\epsilon(f')&=1_{\supp(\epsilon_n(f))}+1_{\supp(\epsilon_n(f'))}\\
&=1_{\supp(\epsilon_n(f))\cup\supp(\epsilon_n(f'))}+1_{\supp(\epsilon_n(f))\cap\supp(\epsilon_n(f'))}\\
&=1_{\supp(\epsilon_n(1_{V\cup V'}))}+1_{\supp(\epsilon_n(1_{V\cap V'}))}\\
&=\epsilon_n(f+f').
\end{align*}
We conclude that for any $l<n-1$ and any $f,f'\in\Lambda_l$, we have $\epsilon_n(f+f')=\epsilon_n(f)+\epsilon_n(f')$. 
	
We now assume that $f'\in\Lambda_l$ and that $f\in\M_l$. Let $(V_j)_j$ be the chain-decomposition of $f$ and let $V':=\supp f'$. It can be checked that $(V_0\cup V',(V_j\cup(V_{j-1}\cap V'))_{j\geq 1})$ is the chain-decomposition of $f+f'$. Again, by the chain-decomposition properties, we obtain $\epsilon_n(f+f')=\epsilon_n(1_{V_0\cup V'})+\sum\limits_{j\geq 1}\epsilon_n(1_{V_j\cup(V_{j-1}\cap V')})$. 
Let us use the result we have just proven to compute
\begin{align*}
\hspace{0,5cm}\epsilon_n(f)+\epsilon_n(f')&=\sum\limits_{j}\epsilon_n(1_{V_j})+\epsilon_n(1_{V'})\\
						&= \epsilon_n(1_{V_0}+1_{V'})+\sum\limits_{j\geq 1}\epsilon_n(1_{V_j})\\
&=\epsilon_n(1_{V_0\cup V'})+\epsilon_n(1_{V_0\cap V'})+\sum\limits_{j\geq 1}\epsilon_n(1_{V_j})\\
&=\epsilon_n(1_{V_0\cup V'})+\epsilon_n(1_{V_0\cap V'}+1_{V_1})+\sum\limits_{j\geq 2}\epsilon_n(1_{V_j})\\
&=\epsilon_n(1_{V_0\cup V'})+\epsilon_n(1_{(V_0\cap V')\cup V_1})+\epsilon_n(1_{V_1\cap V'})+\sum\limits_{j\geq 2}\epsilon_n(1_{V_j})\\
&=\epsilon_n(1_{V_0\cup V'})+\epsilon_n(1_{(V_0\cap V')\cup V_1})+\epsilon_n(1_{V_1\cap V'}+1_{V_2})+\sum\limits_{j\geq 3}\epsilon_n(1_{V_j})\\
&=\epsilon_n(1_{V_0\cup V'})+\epsilon_n(1_{(V_0\cap V')\cup V_1})+\epsilon_n(1_{(V_1\cap V')\cup V_2})+\epsilon_n(1_{V_2\cap V'})+\sum\limits_{j\geq 3}\epsilon_n(1_{V_j}).
\end{align*}

We recall that the chain-decomposition of $f$ is finite and hence, there exists $j\in\N$ such that $1_{V_k\cap V'}=0$ for any $k\geq j$. Repeating the process finitely many times, we get that
\[
\epsilon_n(f)+\epsilon_n(f')=\epsilon_n(1_{V_0\cup V'})+\sum\limits_{j\geq 1}\epsilon_n(1_{V_j\cup(V_{j-1}\cap V')}).
\]
Thus $\epsilon_n(f)+\epsilon_n(f')=\epsilon_n(f+f')$ for any $f\in\bigcup\limits_{l<n-1}M_l$, which ends the proof.
\end{proof}
\begin{cor}
For any $n\in\N$, consider $M_n$ constructed in \autoref{prg:cover}. Let $\epsilon_n: S_{\ll}\longrightarrow M_n$ be the assignment that sends $s$ to $\max\limits_{g\in M_n}\{ g\ll f\}$.

Then $\mathcal{B}_{2^\infty}:=(M_n,\epsilon_n)_n$ is a uniform basis of $\Lsc(]0,1[,\overline{\N})$.
\end{cor}

\begin{prg}
\label{prg:qbases}$\textbf{Uniform bases of size q}.$ We have constructed a uniform basis of $\Lsc(]0,1[,\overline{\N})$. The choice of an equidistant partition of size $1/2^n$ in the construction of $M_n$ (see \autoref{prg:cover}) is arbitrary and one can construct similarly other uniform bases using any supernatural number.
Let $(p_i)_i$ be a sequence of prime numbers and consider the supernatural number $q:=\prod\limits_{i=0}^\infty p_i$. For any $n\in\N$, we define $q_n:=\prod\limits_{i=0}^n p_i$ and we consider an equidistant partition $(x_{k,n})_{k=0}^{q_n}$ of $]0,1[$ of size $1/q_n$ that induces a closed (finite) cover $\mathcal{U}_n:=\{\overline{U_{k,n}}\}_{k=1}^{q_n}$ in the same fashion as in \autoref{prg:cover}, where $U_{k,n}:=]x_{k-1,n};x_{k,n}[$. Similarly, define for any $n\in\N$
\[
\begin{array}{ll} 
\left\{
\begin{array}{ll} 
M_n:=\{f\in\Lsc(]0,1[,\N)\mid f_{|U_{k,n}} \text{ is constant for any }k\in \{1,\dots,q_n\}\}.\\
\epsilon_n: \Lsc(]0,1[,\N)\longrightarrow M_n
 \end{array}
 \right.\\
  \hspace{2,8cm}f\longmapsto\max\limits_{g\in M_n}\{ g\ll f\}
   \end{array}
\]
Then $\mathcal{B}_q:=(M_n,\epsilon_n)_n$ is a uniform basis of $\Lsc(]0,1[,\overline{\N})$ that we refer to as \emph{the uniform basis of size $q$}.
\end{prg}
\subsection{The Cuntz semigroup of continuous functions over one-dimensional \texorpdfstring{$\CW$}{CW} complexes}

This part of the paper aims to generalize the previous process to a larger class of $\Cu$-semigroups. More precisely, we extend the construction of a uniform basis of size $q$, to any $\Cu$-semigroup of the form $\Lsc(X,\overline{\N}^r)$, where $X$ is a compact (Hausdorff) one-dimensional $\CW$ complex (in other word, a finite graph) and $r<\infty$. As a result, we obtain that any concrete Cuntz semigroup coming from $\mathcal{C}(X)\otimes A$, where $X$ is a finite graph and $A$ is a finite-dimensional $\CatCa$-algebra, is uniformly based.

Let us first recall some facts about $\CW$-complexes and we refer the reader to \cite{FP93} for more details. A $\CW$-complex of dimension $n$ is obtained from \textquoteleft gluing\textquoteright\ n-dimensional balls together in a specific way, using $\CW$-complexes of lower dimensions. In the particular case of $n\leq 1$, these topological spaces are nicely characterized: A one-dimensional $\CW$-complex $X$ can be seen as a graph $\mathfrak{c}_X$ where the $0$-dimensional cells are the vertices and the $1$-dimensional cells are the edges. (We precise that in the sequel, we distinguish open and closed cells.) Note that a one-dimensional $\CW$-complex $X$ is compact if and only if its associated graph $\mathfrak{c}_X$ is finite.

\begin{thm}
\label{thm:cpctgraph}
Let $X$ be a compact one-dimensional $\CW$ complex. Then $\Lsc(X,\overline{\N})$ is a uniformly based $\Cu$-semigroup.
\end{thm}

\begin{proof}
Let $X$ be a compact one-dimensional $\CW$-complex. Let us enumerate the cells of $X$ as follows: we denote the $j$-th (clopen) $0$-cell by $c^0_j$, the $j$-th (open) $1$-cell by $c^1_j$ and we define $\mathfrak{c}:=(\{c^0_j\}_1^{l'},\{c^1_j\}_1^l)$ to be the (unique) finite graph associated to $X$. (Observe that a $0$-cell is topologically isomorphic to a single point while an open $1$-cell is topologically isomorphic to the open interval and also that both $l,l'<\infty$.) We recall that $F\subseteq X$ is a closed set if and only if for any $1$-cell $c^1_j$ of $X$, the intersection $F\cap \overline{c^1_j}$ is a closed set of $\overline{c^1_j}$. 

We construct a uniform basis for the case where $X$ is a finite connected graph (in other words, $X$ does not contain any isolated $0$-cells) and the general case will follow. (It is easily seen that $\Lsc(X,\overline{\N})\simeq \underset{i\in I}{\oplus}\Lsc(X_i,\overline{\N})\oplus\overline{\N}^r$ where $I$ is a finite index corresponding to the number of non-trivial connected components $X_i$ in $X$ and $r\in\N$ is the finite number of isolated points in $X$.) 

Let $n\in\N$. First, we construct a monoid $M_n\subseteq \Lsc(X,\N)$ as follows: For each (open) $1$-cell $c^1_j$, consider an equidistant partition of size $1/2^n$, and consider the open sets $\{U_k^j\}_{k=1}^{2^n}$ as in \autoref{prg:cover}. Observe that $\mathcal{U}^j:=\{\overline{U^j_k}\}_{k=1}^{2^n}$ is a closed cover of the closed $1$-cell $\overline{c^1_j}$. Therefore, we obtain a finite closed cover $\mathcal{U}:=\{\{\overline{U^1_k}\}_{k=1}^{2^n},\dots,\{\overline{U^l_k}\}_{k=1}^{2^n}\}$ of $X$. We now consider
\[M_n:=\{f\in\Lsc(X,\N)\mid f_{|U_k^j} \text{ is constant for any } U^j_k\in\mathcal{U}\}.\]
Next, we have to construct an order-preserving super-additive morphism $\epsilon_n:\Lsc(X,\N)\longrightarrow M_n$ satisfying (U2)-(U3)-(U4). As before, we first consider an indicator map $f:=1_V$, where $V$ is an open set of $X$. By the lower-semicontinuity of $f$, we deduce that for any $0$-cell $c_j^0=\{x_j\}$ that belongs to $V$, (that is, $f(x_j)\neq 0$) then there exists an open neighborhood $O_{x_j}$ of $x_j$ such that $O_{x_j}\subset V$. In particular, for any $1$-cell $c^1_k$ which is \textquoteleft glued\textquoteright\ to $\{x_j\}$, in the sense that its closure $\overline{c^1_k}$ contains $x_j$, then there exists an open neighborhood of $x_j$ in $\overline{c^1_k}$ that is contained $V$. Therefore, we can construct $\max\limits_{g\in M_n}\{ g\ll f\}$, in the same fashion as for the open interval, as follows: 

\vspace{0,1cm}(1) For any $1\leq k\leq 2^n$ and any $1\leq j \leq l$, $\left\{\begin{array}{ll}\epsilon_n(f)_{|U^j_k}:=1 \text{ if } \overline{U^j_k}\subseteq V.\\ \epsilon_n(f)_{|U^j_k}:=0 \text{ otherwise}. \end{array}\right.$\\

\vspace{0,15cm}(2) For any $x\in X\setminus (\underset{j=1}{\overset{l}\cup}\underset{k=1}{\overset{2^n}\cup}U^j_k)$, $\left\{\begin{array}{ll}\epsilon_n(f)(x):=1 \text{ if } \epsilon_n(f)_{|U_k}\neq 0 \text{ for all }j,k \text{ such that } x\in \overline{U^j_k}.\\ \epsilon_n(f)(x):=0 \text{ otherwise}. \end{array}\right.$\\
	
Again, using the chain-decomposition properties (see \autoref{prg:chaindcp} and \autoref{prop:chaindcp}) and similar arguments as for the open interval, one can check that we can extend the above construction to any $f\in\Lsc(X,\N)$. Also that $\epsilon_n:S_\ll\longrightarrow M_n$ is an order-preserving super-additive morphism that has the required properties.
\end{proof}

As for the open interval, we can replicate the above construction using any other supernatural number $q:=\prod\limits_{i=0}^\infty p_i$, in order to obtain a uniform basis $\mathcal{B}_q$ of $\Lsc(X,\overline{\N})$ associated to $q$, where $X$ is a compact one-dimensional $\CW$ complex. (See \autoref{prg:qbases}.) We enunciate the result in the following corollary, where we also allow the codomain to be any simplicial $\Cu$-semigroup:

\begin{cor}
\label{cor:qbases}
Let $X$ be a compact one-dimensional $\CW$ complex and let $A$ be a finite dimensional $\CatCa$-algebra. Then $S:=\Cu(\mathcal{C}(X)\otimes A)$ is a  uniformly based $\Cu$-semigroup. More particularly, $S$ has a uniform basis $\mathcal{B}_q$ of size $q$, for any supernatural number $q$.

In particular, the Cuntz semigroups of $\mathcal{C}([0,1])$ and $\mathcal{C}(\T)$ both have uniform bases.
\end{cor}

\begin{proof}
Observe that we naturally have $S\simeq \Lsc(X,\overline{\N}^r)\simeq\underset{r}{\oplus}\Lsc(X,\overline{\N})$. The corollary follows by \autoref{prop:suma}.
\end{proof}

\subsection{The Cuntz semigroup of one-dimensional \texorpdfstring{$\NCCW$}{NCCW} complexes}
We end the example section by studying the Cuntz semigroup of one-dimensional $\NCCW$ complexes. A one-dimensional $\NCCW$ complex is constructed as a pullback of a finite dimensional $\CatCa$-algebra $E$ with $\mathcal{C}([0,1])\otimes F$ (where $F$ is also a finite dimensional $\CatCa$-algebra) which dictates the values at endpoints $0$ and $1$. Inductive limits of these pullbacks is a rather large class within the $\CatCa$-algebras of stable rank one that includes for instance $\AF, \AI$ and $\A\!\T$ algebras. We mention that in \cite{R12} Robert has been able to classify all the unital one-dimensional $\NCCW$ complexes with trivial $\K_1$-group and inductive limits of such building blocks by means of the functor $\Cu$. 

Here, we aim to prove that the Cuntz semigroup of a one-dimensional $\NCCW$ complex, under a suitable hypothesis on the morphisms defining the pullback, has a uniform basis. The strategy adopted is to \textquoteleft extract\textquoteright\ such a uniform basis by using any uniform basis that we have built for the interval. This process involves the notion of \emph{$\Cu$-subsemigroups} that we recall now.\\

\vspace{-0,2cm}$\hspace{-0,34cm}\bullet\,\,\Cu\textbf{- subsemigroups and compact-containment relation}.$
Let $S$ be a $\Cu$-semigroup. We recall a $\Cu$-ideal of $S$ consists of a submonoid $I$ that is order-hereditary and closed under suprema of increasing sequences. We studying substructures of a $\Cu$-semigroup, it is often convenient to consider $\Cu$-ideals. For instance, $\Cu$-ideals allow the construction of quotients and they preserve the compact-containment relation, in the sense that the canonical inclusion $i:I\lhook\joinrel\longrightarrow S$ of an ideal $I$ into $S$ is a $\Cu$-morphism. In particular, we have $I_\ll\subseteq S_\ll\cap I$. 
Nevertheless, it is not true that the Cuntz semigroup of a one-dimensional $\NCCW$ complex can be seen as an ideal of $\Lsc([0,1],\overline{\N}^r)$ since it might not be order-hereditary. We hence have to work with \emph{$\Cu$-subsemigroups}, a weaker substructure introduced in \cite{TV21}.

\begin{dfn} (\cite[Definition 4.1]{TV21})
\label{dfn:subsemigp}
Let $S$ be a $\Cu$-semigroup. We say that a subset $S'\subseteq S$ is a \emph{$\Cu$-subsemigroup of $S$} if, $S'$ is a $\Cu$-semigroup and the canonical inclusion $i:S'\lhook\joinrel\longrightarrow S$ is a $\Cu$-semigroup. 
\end{dfn}

\begin{rmk}
One might have tried to define a $\Cu$-subsemigroup as a subset which is itself a $\Cu$-semigroup. However, one has to be cautious and notice that the compact-containment relation can be altered. Indeed, it could happen that $S'_\ll\nsubseteq S_\ll\cap S'$. (Therefore, two elements could compare for $\ll$ in $S'$ but not in $S$.) For instance, set  $S:= \overline{\N}$ and $S':=\{0,\infty\}$. Then $ S'_\ll=\{0,\infty\}$ and $S_\ll\cap S'=\{0\}$. Thus $S'_\ll\supsetneq S_\ll\cap S'$.

Also, one could consider submonoids that are closed under suprema of increasing sequences and $\ll$-hereditary, in the sense that if $x\ll y$ with $x\in S$ and $y\in S'$, then $x\in S'$. This notion seems stronger than the one define in \autoref{dfn:subsemigp} (and weaker than the notion of ideal) but does not seem relevant for our case.
\end{rmk}

The next proposition shows that any $\Cu$-subsemigroup $S'$ of a uniformly based $\Cu$-semigroup $S$, under suitable stability condition, also admits a uniform basis that can be obtained by \textquoteleft restricting\textquoteright\ the uniform basis of $S$.

\begin{prop}
\label{prop:llstability}
Let $S$ be a uniformly based $\Cu$-semigroup with uniform basis $\mathcal{B}:=(M_n,\epsilon_n)_{n}$. Let $S'\subseteq S$ be a $\Cu$-subsemigroup of $S$. Assume that $S'$ is $\mathcal{B}$-stable, in the sense that $\epsilon_n(S'_\ll)\subseteq S'_\ll$ for any $n\in\N$. 
Then $(\epsilon_n(S'_\ll),{\epsilon_n}_{\mid S'_\ll})_{n}$ is a uniform basis of $S'$. 
\end{prop}

\begin{proof}
It is left to the reader to check that $\mathcal{B}$-stability precisely allows us to define the restriction of $\epsilon_n$ to $S'_\ll$ and ensures that the monoids $(\epsilon_n(S'_\ll))_n$ together with the maps $({\epsilon_n}_{\mid S'_\ll})_n$ satisfy axioms (U1)-...-(U4).
\end{proof}

$\hspace{-0,34cm}\bullet\,\,\textbf{One-dimensional} \NCCW \textbf{complexes}.$
Let $E,F$ be finite dimensional $\CatCa$-algebras and let $\phi_0,\phi_1:E\longrightarrow$ be two $^*$-homomorphisms. We define a \emph{non-commutative CW complex of dimension one}, abbreviated one-dimensional $\NCCW$ complex, as the following pullback: 
\[
\xymatrix{
A\ar[r]^{}\ar[d]_{}& C([0,1],F)\ar[d]^{(ev_0,ev_1)}\\
E\ar[r]_{(\phi_0,\phi_1)}& F\oplus F
}
\]
We write such a pullback as $A:=A(E,F,\phi_0,\phi_1)$. The Cuntz semigroups of one-dimensional $\NCCW$ complexes have been computed in \cite[Section 4.2]{APS11} as follows: Write $\alpha:=\Cu(\phi_0)$ and $\beta:=\Cu(\phi_1)$. Then
\[
\Cu(A)\simeq\{(f,m)\in\Lsc([0,1],\Cu(F))\oplus\Cu(E) \mid \alpha(m)=f(0) \text{ and } \beta(m)=f(1)\}.
\]

\begin{prop}
\label{prop:subseminccw}
Let $A:=A(E,F,\phi_0,\phi_1)$ be a one-dimensional $\NCCW$ complex. If either $\Cu(\phi_0)$ or $\Cu(\phi_1)$ is an order-embedding, then $\Cu(A)$ can be identified with a $\Cu$-subsemigroup of $\Lsc([0,1],\Cu(F))$ that is also denoted by $\Cu(A)$. 

Furthermore, $\Cu(A)$ is $\mathcal{B}_q$-stable for any uniform basis $\mathcal{B}_q$ of $\Lsc([0,1],\Cu(F))$ of size $q$ constructed in \autoref{cor:qbases}.
\end{prop}

\begin{proof}
Let $A:=A(E,F,\phi_0,\phi_1)$. We denote $\alpha:=\Cu(\phi_0), \beta:=\Cu(\phi_1)$ and $S:=\Cu(E), T:=\Cu(F)$. Consider the following $\CatPoM$-morphism:
\[
\begin{array}{ll}
\iota:\Cu(A)\longrightarrow \Lsc([0,1],T)\\ 
\hspace{0,45cm} (f,m)\longmapsto f
\end{array}
\] 
Let $(f,m),(g,n)\in \Cu(A)$ be such that $\iota((f,m))\leq \iota((g,n))$. In particular, we know that $f(0)\leq g(0)$ and $f(1)\leq g(1)$. Equivalently, $\alpha(m)\leq \alpha(n)$ and $\beta(m)\leq \beta(n)$. Using that either $\alpha$ or $\beta$ is an order-embedding, we deduce that $m\leq n$ in $S$. We obtain that $(f,m)\leq (g,n)$ and hence, $\iota$ is an order-embedding. Further, it is easily seen that $\iota$ respects suprema of increasing sequences. Let us prove that $\iota$ respects the compact-containment relation. First, we can identify 
\[
\Cu(A)\overset{\iota}{\simeq}\{f\in\Lsc([0,1],T) \mid f(0)=\alpha(s) \text{ and } f(1)=\beta(s) \text{ for some } s\in S\}.
\]
Let $g,h\in\Cu(A)$ be such that $g\ll h$ in $\Cu(A)$. We have to check that $g\ll h$ in $\Lsc([0,1],T)$.
Let $(h_k)_k$ be an increasing sequence in $\Lsc([0,1],T)$ such that $\sup\limits_k h_k=h$. Let $k\in\N$ and consider $s_k:=\max\{s\in S\mid\alpha(s)\leq h_k(0)\text{ and }\beta(s)\leq h_k(1)\}$. Now construct $h'_k\in \Cu(A)$ as follows:
\[
\left\{\begin{array}{ll}
h'_k(t):=h_k(t) \text{ for any } t\in ]0,1[\\
h'_k(0):=\alpha(s_k)\\ 
h'_k(1):=\beta(s_k)
\end{array}
\right.
\] 
It is immediate that $(h'_k)_k$ is an increasing sequence of $\Cu(A)$ such that $h'_k\leq h_k$ for any $k\in\N$ and such that $\sup\limits h'_k=h$. Therefore, there exists $n\in\N$ such that $g\leq h'_n\leq h_n$ from which we conclude that $g\ll h$ in $\Lsc([0,1],T)$ and hence that $\iota$ is a $\Cu$-morphism.
It follows that $\Cu(A)$ can be identified with a $\Cu$-subsemigroup of $\Lsc([0,1],T)$, through $\iota$, that we also denote by $\Cu(A)$. 

We now have to check that $\Cu(A)$ is $\mathcal{B}_q$-stable. Without loss of generalities we can assume that $q=2^\infty$.
Let $g\in\Cu(A)_\ll$ and let $n\in\N$. We construct the finite closed cover $\mathcal{U}:=\{\overline{U_k}\}_1^{2^{n}}$ of $[0,1]$ as in \autoref{prg:cover}. Then either $\epsilon_n(g)(0)=0$, or else $\epsilon_n(g)(0)=g(0)(\neq 0)$ precisely when $\supp g$ contains $\overline{U_1}$. A similar argument can be made for $\epsilon_n(g)(1)$ and we deduce that $\epsilon_n(\Cu(A)_\ll)\subseteq \Cu(A)_\ll$. 
\end{proof}

\begin{cor}
The Cuntz semigroup of any one-dimensional $\NCCW$ complex $A(E,F,\phi_0,\phi_1)$ is uniformly based whenever $\Cu(\phi_0)$ or $\Cu(\phi_1)$ is an order-embedding. 
\end{cor}

We next exhibit specific one-dimensional $\NCCW$ complexes that have been considered in the past which satisfy the assumption of \autoref{prop:subseminccw}. A fortiori, their Cuntz semigroups are uniformly based and we explicitly compute such uniform bases. Note that these uniform bases will be involved in a forthcoming manuscript. (See \cite{C22}.)

\begin{prg}\label{prg:ET}
\textbf{Elliott-Thomsen dimension drop algebras over the interval}.

Elliott-Thomsen dimension-drop interval algebras are one of the first one-dimensional $\NCCW$ complexes that have been constructed. Roughly, they tend to mimic the construction of the circle as a one-dimensional $\CW$-complex. More concretely, these algebras are defined as \[I_r:=A(\mathbb{C}\oplus\mathbb{C},M_r,\pi_0\otimes 1_r,\pi_1\otimes 1_r)\] where $r\in\N$ is a natural number and $\pi_0,\pi_1:\mathbb{C}\oplus\mathbb{C}\longrightarrow \mathbb{C}$ are the respective projections on each component of the direct sum. 
We mention that
\[
\K_0(I_r)\simeq \Z \hspace{3cm} \K_1(I_r)\simeq \Z/r\Z. 
\]
Further, observe that $\Cu(\pi_0\otimes 1_r)$ and $\Cu(\pi_0\otimes 1_r)$ are $\Cu$-morphisms from $\overline{\N}$ to $\overline{\N}$ and they both send $1\longmapsto r$. Thus both are order-embeddings. Using \autoref{prop:subseminccw} and more precisely the morphism $\iota$ constructed in the proof, we compute that
\begin{align*}	 
\Cu(I_r)&\simeq \{f\in \Lsc([0,1],\overline{\N}) \mid f(0),f(1)\in r\overline{\N}\}\\ 
	&\simeq \{f\in \Lsc([0,1],\frac{1}{r}\overline{\N}) \mid f(0),f(1)\in \overline{\N}\}.
\end{align*}
Finally, one can explicitly construct a uniform basis of $\Cu(I_r)$ using any uniform basis of $\Lsc([0,1],\overline{\N})$ of size $q$; see \autoref{prg:qbases} and \autoref{prop:llstability}.
\end{prg}

\begin{prg}\label{prg:EK}
\textbf{Evans-Kishimoto folding interval algebras}.

Evans-Kishimoto folding interval algebras are a generalization of Elliott-Thomsen dimension drop interval algebras; see \autoref{prg:ET}. That said, they have been first considered in \cite{EK91} just a few years before the dimension drop algebras. 
In this paragraph, we recall the construction of these $\CatCa$-algebras and expose some of their properties. (We refer the reader to \cite[(2.14)]{EK91} for the original construction.)

Let $r$ be a natural number and denote the full matrix algebra of size $r$ by $M_r$. Let $p_e$ be a projection of $M_r$ and write $e:=\rank(p_e)$. For any $l\in\N_*$, we consider the following pullback:
\[
\xymatrix{
\mathcal{I}^l_{r,e}\ar[r]^{\pi_1}\ar[d]_{\pi_2}& C([0,1],\underset{1}{\overset{l}\otimes} M_r)\ar[d]^{(ev_0, ev_1)}\\
(\underset{1}{\overset{l-1}\otimes} M_r)\oplus (\underset{1}{\overset{l-1}\otimes} M_r)\ar[r]_{(i_0^l, i_1^l)}& (\underset{1}{\overset{l}\otimes} M_r)\oplus (\underset{1}{\overset{l}\otimes} M_r)
}
\]
where $i_0^l,i_1^l: (\underset{1}{\overset{l-1}\otimes} M_r)\oplus (\underset{1}{\overset{l-1}\otimes} M_r)\longrightarrow \underset{1}{\overset{l}\otimes} M_r$ are injective $^*$-homomorphisms constructed by induction. We refer the reader to  \cite[(2.11)/(2.12)]{EK91} for more details and we mention that the original construction involves two projections $E_1$ and $E_2$, that are given in our case by $E_1:=e$ and $E_2:= 1-e$. Let us precise that in the case $l=1$, we have that $\mathcal{I}^1_{r,r}= I_r$ is the Elliott-Thomsen dimension-drop interval algebra.  We also mention that for any $l\in\N$, we have 
\[
\K_0(\mathcal{I}^l_{r,e})\simeq \Z\hspace{3cm} \K_1(\mathcal{I}^l_{r,e})\simeq \Z/r\Z.\]
Further, observe that $\Cu(i_0^l)$ and $\Cu(i_1^l)$ are $\Cu$-morphisms from $\overline{\N}$ to $\overline{\N}$ and they both send $1\longmapsto r$; see \cite[Lemma 2.2]{EK91}. Thus both are order-embeddings. Using \autoref{prop:subseminccw} and more precisely the morphism $\iota$ constructed in the proof, we compute that
\begin{align*}	 
\Cu(\mathcal{I}^l_{r,e})&\simeq \{f\in \Lsc([0,1],\overline{\N}) \mid f(0),f(1)\in r\overline{\N}\}\\
	&\simeq \{f\in \Lsc([0,1],\frac{1}{r^l}\overline{\N}) \mid f(0), f(1)\in \frac{r}{r^l}\overline{\N}\}.
\end{align*}

Let us explicitly construct a uniform basis of $\Cu(\mathcal{I}^l_{r,e})$. Let $q:=\prod\limits_{i=0}^\infty p_i$ a supernatural number. For any $n\in\N$, define $q_n:=\prod\limits_{i=0}^n p_i$ and $\mathcal{U}_n:=\{\overline{U_{k,n}}\}_{k=1}^{q_n}$ as in \autoref{prg:qbases}. Now, let $\mathcal{B}_q=(M_n,\epsilon_n)_n$ be the uniform basis of $\Lsc([0,1],\overline{\N})$ of size $q$. By \autoref{prop:llstability}, we know that $\mathcal{B}_q$ induces a uniform basis $\mathcal{B}'_q:=(M'_n,\epsilon_n')_n$ of $\Cu(\mathcal{I}^l_{r,e})$ that we compute as follows:
\[
\begin{array}{ll} 
\left\{
\begin{array}{ll} 
M'_n:=\{f\in\Lsc(X,\frac{1}{r^l}\N)\mid f_{|U_{k,n}} \text{ is constant for any }k\in \{1,\dots,q_n\} \text{ and } f(0), f(1)\in \frac{r}{r^l}\overline{\N}\}.\\
\epsilon'_n: \{f\in\Lsc(X,\frac{1}{r^l}\N)\mid f(0), f(1)\in \frac{r}{r^l}\overline{\N}\}\longrightarrow M'_n
 \end{array}
 \right.\\
  \hspace{6,5cm}f\longmapsto\max\limits_{g\in M'_n}\{ g\ll f\}
   \end{array}
\]
Naturally, we refer to $\mathcal{B}'_q:=(M'_n,\epsilon_n')_n$ as the uniform basis of $\Cu(\mathcal{I}^l_{r,e})$ of size $q$.
\end{prg}

\begin{rmk}
Many one-dimensional $\NCCW$ complexes seem to be constructed in such a way that $\phi_0,\phi_1$ are injective $^*$-homomorphisms; see \cite[Section 4]{R12} and \cite[Section 4]{APS11}. It would be interesting to see whether this is a sufficient condition to conclude that $\Cu(\phi_0),\Cu(\phi_1)$ are order-embeddings. (It seems reasonable enough, for e.g., we already know that in this case $\K_0(\phi_0), \K_0(\phi_1)$ are also injective maps.)
\end{rmk}

\section{Applications}
\subsection{Construction of \texorpdfstring{$\Cu$}{Cu}-semimetrics on \texorpdfstring{$\Hom_{\Cu}(S,T)$}{HomCu(S,T)} for uniformly based Cuntz semigroups}
Let $S$ be a uniformly based $\Cu$-semigroup with uniform basis $\mathcal{B}:=(M_n,\epsilon_n)_n$ and let $T$ be any $\Cu$-semigroup. In the sequel, we see that $\mathcal{B}$ induces a semimetric $dd_{\Cu,\mathcal{B}}$ on $\Hom_{\Cu}(S,T)$.  Let us remind that a semimetric on a set $X$ is a symmetric function defined on $X\times X$ with values in $\R_+$ that satisfies the identity of indiscernibles but need not satisfy the triangle inequality. In the specific case of $S=\Lsc(X,\overline{\N})$, where $X$ is a compact one-dimensional $\CW$ complex, we see that any two uniform bases $\mathcal{B}_q, \mathcal{B}_{q'}$ of respective size $q$ and $q'$ induce topologically equivalent semimetrics. In particular, we see that these $\Cu$-semimetrics, that we sometimes refer to as \emph{discrete $\Cu$-metrics}, are generalizing $\Cu$-metrics that had been introduced in the past for the specific cases of the interval and the circle (see e.g. \cite{RS09}, \cite{JSV18}). 

In order to define such a metric, we start by observing that any two $\Cu$-morphisms $\alpha,\beta:S\longrightarrow T$ are equal if and only if $\alpha\underset{M_n}{\simeq}\beta$ for any $n\in\N$. (Combine axiom (O2) together with the fact that $\bigcup M_n$ is dense of $S$.) This allows us to define a $\Cu$-semimetric on $\Hom_{\Cu}(S,T)$ associated to $\mathcal{B}$ as follows:

\begin{dfn}
\label{dfn:semimetric}
Let $S$ be a uniformly based $\Cu$-semigroup. Let $\mathcal{B}:=(M_n,\epsilon_n)_n$ be a uniform basis of $S$. Let $\alpha,\beta:S\longrightarrow T$ be two $\Cu$-morphisms. We define
\[\hspace{-0,8cm}\hspace{0,3cm} dd_{\Cu,\mathcal{B}}(\alpha,\beta):= \inf\limits_{n\in\N} \{ \frac{1}{n} \mid \alpha\underset{M_n}{\simeq}\beta \}.
\]  
If the infimum defined does not exist, we set the value to $\infty$.
We refer to $dd_{\Cu,\mathcal{B}}$ as the \emph{discrete $\Cu$-semimetric} associated to the uniform basis $\mathcal{B}$. 
\end{dfn}
Observe that whenever convenient, we can \textquoteleft rescale\textquoteright\ a discrete $\Cu$-semimetric associated to a uniform basis by replacing $\frac{1}{n}$ by $\frac{1}{q(n)}$, where $q:\N\longrightarrow \N$ is any strictly increasing map. 

In general, it might be difficult to compare two $\Cu$-morphisms on (infinite) sets such as the 
monoids $M_n$. In most cases, we tend to use comparison of $\Cu$-morphisms on finite sets. 
As observed in \autoref{prg:cover} and in many proofs of \autoref{sec:A}, if a $\Cu$-semigroup $S$ has a uniform basis $(M_n,\epsilon_n)_n$ and a chain-generating set $\Lambda$ such that $\Lambda\cap M_n$ is a finite chain-generating set of $M_n$, then it is enough to work with elements of $\Lambda_n$. For instance, any two $\Cu$-morphisms $\alpha,\beta:S\longrightarrow T$ compare on $M_n$ if and only if they compare on $\Lambda_n$. 
In this case, we may speak of a \emph{finite uniform basis} and adjust the $\Cu$-semimetric defined above, as detailled below.\\

\vspace{-0,2cm}$\hspace{-0,34cm}\bullet\,\,\textbf{Finite uniform bases}.$ Let $S$ be a uniformly based $\Cu$-semigroup with uniform basis $(M_n,\epsilon_n)$. Assume that $S$ has a chain-generating set $\Lambda$ that induces finite chain-generating sets $\Lambda_n:=\Lambda\cap \M_n$ of $M_n$ for any $n\in\N$. Then we say that $S$ has a \emph{finite uniform basis} that we denote $(\Lambda_n,\epsilon_n)$. 

In this case, we can reformulate
\vspace{-0,2cm}\[\hspace{-0,8cm}\hspace{0,3cm} dd_{\Cu,\mathcal{B}}(\alpha,\beta) = \inf\limits_{n\in\N} \{ \frac{1}{n} \mid \alpha\underset{\Lambda_n}{\simeq}\beta \}.
\] 

We now show that any two finite uniform bases induce topologically equivalent semimetrics. As a consequence, we next deduce that all the uniform bases $\mathcal{B}_q$ of size $q$ obtained earlier in \autoref{cor:qbases} and in \autoref{prg:EK} are inducing equivalent $\Cu$-semimetrics.

\begin{lma}
\label{lma:lmacomparison}
Let $S$ be a uniformly based $\Cu$-semigroup and $T$ be any $\Cu$-semigroup. Let $\mathcal{B}:=(\Lambda_n,\epsilon_n)_n$ and $\mathcal{B}':=(\Lambda'_n,\epsilon'_n)_n$ be two finite uniform bases of $S$.

Then, for any $n\in\N$ there exists a minimal $p(n)\in\N$ such that for any two $g',g\in \Lambda_n$ with $g'\ll g$, we can find $h',h$ in $\Lambda'_{p(n)}$ with $g'\ll h'\ll h\ll g$. 
\end{lma}

\begin{proof}
Let $n\in\N$ and let $g',g\in \Lambda_n$ with $g'\ll g$. Since $(\epsilon'_m(g))_{m>n}$ is $\ll$-increasing towards $g$, we can find $m>n$ such that $g'\ll \epsilon'_{m-1}(g)\ll \epsilon'_{m}(g)\ll g$. By finiteness of the set $\Lambda_n$, the lemma is almost immediate and left to the reader to check.
\end{proof}

Let $\mathcal{B}:=(\Lambda_n,\epsilon_n)_n$ and $\mathcal{B}':=(\Lambda'_n,\epsilon'_n)_n$ be two finite uniform bases of a $\Cu$-semigroup $S$. As a consequence of \autoref{lma:lmacomparison}, we can construct two maps as follows:
\[
\begin{array}{ll}
 \hspace{1cm}p_{B'B}:\N\longrightarrow \N \hspace{3cm} q_{B'B}:\N\longrightarrow \N\\
 \hspace{2,1cm} n\longmapsto p(n) \hspace{3,55cm}  m\longmapsto \max\{p^{-1}(\{0,\dots,m\})\}
\end{array}
\]
\begin{prop}
\label{prop:comparisonequiv}
Let $S$ be a uniformly based $\Cu$-semigroup that has two finite uniform bases $\mathcal{B}:=(\Lambda_n,\epsilon_n)_n$ and $\mathcal{B}':=(\Lambda'_n,\epsilon'_n)_n$. Consider the maps $p_{B'B}$ and $q_{B'B}$ constructed above. Then:

(i) $q_{B'B}$ is $\leq$-increasing towards $\infty$. Moreover $q_{B'B}(\N_*)\subseteq \N_*$. 

(ii) Let $\alpha,\beta:S\longrightarrow T$ be two $\Cu$-morphisms. If $\alpha\underset{\Lambda'_{n}}{\simeq}\beta$ then $\alpha\underset{\Lambda_{q_{B'B}(n)}}{\simeq}\beta$. In particular,
\[q_{B'B}(\sup\limits_{n\in\N}\{n \mid \alpha\underset{\Lambda'_n}{\simeq}\beta\})\leq \sup\limits_{n\in\N} \{n\mid\alpha\underset{\Lambda_n}{\simeq}\beta\}.\]
\end{prop}

\begin{proof}
(i) follows from the fact that $p_{B'B}$ is a $\leq$-increasing map and that $p(0)=0$. 

(ii) Now assume that $\alpha\underset{\Lambda'_{n}}{\simeq}\beta$ for some $n\in\N$. Since $\Lambda'_n$ is a chain-generating set of $M_n$, we know that $\alpha\underset{M'_{n}}{\simeq}\beta$. Moreover, for any $m\leq q_{B'B}(n)$ and any $g',g\in \Lambda_m$ with $g'\ll g$, then there exist $h',h\in M'_n$ such that $g'\ll h'\ll h\ll g$. We deduce that $\alpha(g') \ll \alpha(h')\leq \beta (h)\ll \beta (g)$ and $\beta(g') \ll \beta(h')\leq \alpha (h)\ll \alpha (g)$. It follows that $\alpha(g'),\beta(g')\leq \alpha(g),\beta(g)$ for any $g',g \in \Lambda_m$ with $g'\ll g$ and $m\leq q_{B'B}(n)$.
\end{proof}

\begin{prop}
Let $S$ be a uniformly based $\Cu$-semigroup and let $\mathcal{B}:=(\Lambda_n,\epsilon_n)_n, \mathcal{B}':=(\Lambda'_n,\epsilon'_n)_n$ be two uniform bases of $S$. Let $T\in \Cu$ and let $\alpha:S\longrightarrow T$ be a $\Cu$-morphism. Then for any $n\in\N_*$, we have that 
\[
B_{r''_n}(\alpha,dd_{\Cu,\mathcal{B}})\subseteq B_{1/n}(\alpha,dd_{\Cu,\mathcal{B}'}) \text{ and } B_{1/n}(\alpha,dd_{\Cu,\mathcal{B}'})\subseteq B_{r'_n}(\alpha,dd_{\Cu,\mathcal{B}})
\]
where $r'_n:=1/q_{B'B}(n)$ and ${r''_n}^{-1}\in q_{BB'}^{-1}(\{m\geq n\})$.
\end{prop}

\begin{proof}
The second inclusion is almost immediate: Let $n\in\N_*$ and let $\beta\in B_{1/n}(\alpha,dd_{\Cu,\mathcal{B}'})$. By \autoref{prop:comparisonequiv}, we obtain that $q_{B'B}(n)>0$ and that $\alpha\underset{\Lambda_{q_{B'B}(n)}}{\simeq}\beta$. Thus $\beta\in B_{r'_n}(\alpha,dd_{\Cu,\mathcal{B}})$.  

We now prove the first inclusion: Since $n>0$ and $q_{BB'}$ is increasing towards $\infty$, we know that   $q_{BB'}^{-1}(\{m\geq n\}$ is a non-empty subset of $\N_*$. Take any $n''\in q_{BB'}^{-1}(\{m\geq n\})$ and write $r''_n:=1/{n''}$. Again by \autoref{prop:comparisonequiv}, we have that $\alpha\underset{\Lambda_{q_{FF'}(n'')}}{\simeq}\beta$, for any $\beta\in B_{r''_n}(\alpha,dd_{\Cu,\mathcal{B}})$. Observe that we have chosen $n''$ such that $q_{BB'}(n'')\geq n$, so we conclude that $\alpha\underset{\Lambda_n}{\simeq}\beta$. That is, $\beta\in B_{1/n}(\alpha,dd_{\Cu,\mathcal{B}'})$.
\end{proof}

\begin{prg}$\Cu\textbf{\!-semimetrics associated to a supernatural number q}.$ Let $q:=\prod\limits_{i=0}^\infty p_i$ be a supernatural number. 
Let $S$ be a $\Cu$-semigroup of one the following types:
\[
\begin{array}{ll}
S_1=\Lsc(X,\overline{\N}^r).\\
S_2= \{f\in \Lsc([0,1],\frac{1}{r^l}\overline{\N}) \mid f(0), f(1)\in \frac{r}{r^l}\overline{\N}\}.
\end{array}
\]
We have seen that $S$ admits a uniform basis $\mathcal{B}_q:=(M_n,\epsilon_n)_n$ of size $q$. We observe that $({S_1})_\ll$ and $({S_2})_\ll$ admit the respective chain-generating sets:
\[
\begin{array}{ll}
\Lambda_1=\Lsc(X, \{0,1\}^r).\\
\Lambda_2= \{f\in \Lsc([0,1],\{\frac{j}{r^l}\}_{j=0}^{r}) \mid f(0), f(1)\in \{0,\frac{r}{r^l}\}\}.
\end{array}
\]
It is immediate that $\Lambda_n:=\Lambda_i\cap M_n$ is a finite chain-generating set of $M_n\subseteq ({S_i})_\ll$, for any $n\in\N$ and $i=1,2$. Therefore $S$ admits a finite uniform basis $(\Lambda_n,\epsilon_n)_n$ of size $q$ that we also denote by $\mathcal{B}_q$. 
Let $T$ be a $\Cu$-semigroup and let $\alpha,\beta:S\longrightarrow T$ be two $\Cu$-morphisms. Write $q_n:=\prod\limits_{i=0}^n p_i$. 

We define \emph{the $\Cu$-semimetric associated to the supernatural number $q$} as follows: 
\[dd_{\Cu,\mathcal{B}_q}(\alpha,\beta):= \inf\limits_{n\in\N} \{ \frac{1}{q_n} \mid \alpha\underset{\Lambda_n}{\simeq}\beta \}.
\]

In the case that $q:=p^\infty$ (with $p<\infty$) is a supernatural number of infinite type, we obtain
\[dd_{\Cu,\mathcal{B}_q}(\alpha,\beta):= \inf\limits_{n\in\N} \{ \frac{1}{p^n} \mid \alpha\underset{\Lambda_n}{\simeq}\beta \}.
\]
\end{prg}

Note that when $S=\Lsc(X,\overline{\N})$, with $X$ being either the circle or the interval, a metric has already been defined on $\Hom_{\Cu}(S,T)$, where $T\in\Cu$. (See e.g. \cite{RS09} or \cite{JSV18}.) We compare this metric with the discrete $\Cu$-semimetrics that we have obtained and we show that they are equivalent.

For any open set $V$ of $X$, and any $r>0$, we define an \emph{$r$-open neighborhood of $V$}, that we write $V_r:=\underset{x\in V}{\cup}B_r(x)$. 
Now, for any two $\Cu$-morphisms $\alpha, \beta:\Lsc(X,\overline{\N})\longrightarrow T$, we write
\[  d_{\Cu}(\alpha,\beta):=\inf \{ r>0\mid \forall V\in\mathcal{O}(X), \alpha(1_{V})\leq\beta(1_{V_{r}})  \text{ and }  \beta(1_{V})\leq\alpha(1_{V_{r}}) \} 
\] where $V_r$ is an $r$-open neighborhood of $V$ and $\mathcal{O}(X):=\{$Open sets of $X\}$. This defines metric that we refer to as \emph{the $\Cu$-metric}.

\begin{prop} 
\label{prop:dcuddcu}
Let $X$ be the circle or the interval and let $T$ be $\Cu$-semigroup. Let $\alpha,\beta:\Lsc(X,\overline{\N})\longrightarrow T$ be $\Cu$-morphisms. Let $\mathcal{B}_{2^\infty}:=(\Lambda_n,\epsilon_n)_n$ be the finite uniform basis of $\Lsc(X,\overline{\N})$ associated to $2^\infty$.

(i) If $d_{\Cu}(\alpha,\beta)\leq 1/2^n$, then $dd_{\Cu,\mathcal{B}_{2^\infty}}(\alpha,\beta)\leq 1/2^n$.

(ii) If $dd_{\Cu,\mathcal{B}_{2^\infty}}(\alpha,\beta)\leq 1/2^n$, then $d_{\Cu}(\alpha,\beta)\leq 2/2^n\leq 1/2^{n-1}$. 
\end{prop}

\begin{proof}
Let $n\in\N$. We construct the finite closed cover $\mathcal{U}:=\{\overline{U_k}\}_1^{2^{n}}$ of $X$ as in \autoref{prg:cover}. 

(i) Assume that $d_{\Cu}(\alpha,\beta)\leq 1/2^n$. Let $f',f\in\Lambda_n$ such that $f'\ll f$. Write $V':=\supp f', V:=\supp f$. Both $V,V'$ have a finite number of (open) connected components. We first assume that $V,V'$ are connected open sets of $X$ and we repeat the process finitely many times to obtain the result. We have that:
\[
V':=U_{l'}\cup(\bigcup\limits_{k=l'+1}^{r'-1}\overline{U_{k}})\cup U_{r'} \hspace{2cm}V:=U_{l}\cup(\bigcup\limits_{k=l+1}^{r-1}\overline{U_{k}})\cup U_{r}
\]
for some $l\leq l'\leq r'\leq r$. Since $1_{V'}\ll 1_V$, either $V=V'=X$ or else $l<l'\leq r'<r$. In both cases, we observe that $1_{V'}\ll 1_{V'_{1/2^n}}\leq 1_V$. Thus we deduce that $\alpha(f')\leq \beta(f)$ and $\beta(f')\leq \alpha(f)$, that is $\alpha\underset{\Lambda_n}{\simeq}\beta$.

(ii) Conversely, assume that $dd_{\Cu,\mathcal{B}_{2^\infty}}(\alpha,\beta)\leq 1/2^n$. Let $V$ be an open set of $X$ and consider $f:=1_{V}$. Let us construct recursively the following $f'\in\Lambda_n$:

	\vspace{0,2cm}(1) For any $1\leq k\leq 2^n$, $\left\{\begin{array}{ll}{f'}_{|U_k}:=1 \text{ if } U_k\cap V\neq\emptyset.\\ {f'}_{|U_k}:=0 \text{ otherwise}. \end{array}\right.$\\
	
	\vspace{0,2cm}(2) For any $x\in X\setminus (\underset{k=1}{\overset{2^n}\cup}U_k)$, put $f'(x):=f(x)$.\vspace{0,2cm}\\
Write $V':=\supp f'$. From construction, we have that $f'\in\Lambda_n$. Furthermore, we observe that $V\subseteq V'\subseteq V_{1/2^n}\subseteq {V'}_{1/2^n}\subseteq V_{2/2^n}$. Since $1_{V'},1_{{V'}_{2/2^n}}$ are elements of $\Lambda_n$ such that $1_{V'}\ll 1_{{V'}_{1/2^n}} $, we deduce that
\[
	\left\{
	   	\begin{array}{ll}
   			\alpha(1_{V})\leq \alpha(1_{V'})\leq \beta(1_{{V'}_{1/2^n}})\leq \beta(1_{V_{2/2^n}})\\
   			\beta(1_{V})\leq \beta(1_{V'})\leq \alpha(1_{{V'}_{1/2^n}})\leq \alpha(1_{V_{2/2^n}})
      \end{array}
	\right.
\]
which ends the proof.
\end{proof}

\begin{rmk}
We easily deduce that $dd_{\Cu,\mathcal{B}_{2^\infty}}(\alpha,\beta)\leq d_{\Cu}(\alpha,\beta)\leq 2 dd_{\Cu,\mathcal{B}_{2^\infty}}(\alpha,\beta)$. As a consequence, we get that $dd_{\Cu,\mathcal{B}_{2^\infty}}(\alpha,\beta)\leq 2 ( dd_{\Cu,\mathcal{B}_{2^\infty}}(\alpha,\gamma)+dd_{\Cu,\mathcal{B}_{2^\infty}}(\gamma,\beta))$ for any $\Cu$-morphism $\gamma:\Lsc(X,\overline{\N})\longrightarrow T$. We say that $dd_{\Cu,\mathcal{B}_{2^\infty}}$ satisfies the \emph{$2$-relaxed triangle inequality} and such a semimetric is sometimes referred to as a \emph{nearmetric}.
\end{rmk}

\emph{Conjecture:} Let $X$ be the circle or the interval. Let $\alpha, \beta:\Lsc(X,\overline{\N})\longrightarrow T$ be $\Cu$-morphisms. Let $q:=p^\infty$ (with $p<\infty$) be a supernatural number of infinite type. Then \[dd_{\Cu,\mathcal{B}_{p^\infty}}(\alpha,\beta)\leq d_{\Cu}(\alpha,\beta)\leq p dd_{\Cu,\mathcal{B}_{p^\infty}}(\alpha,\beta)\]
and $dd_{\Cu,\mathcal{B}_{p^\infty}}$ is a nearmetric that satisfies the $p$-relaxed triangle inequality.

\subsection{Classification of unitary elements of \texorpdfstring{$\AF$}{AF}-algebras by means of the Cuntz semigroup}
We conjecture that the Cuntz semigroup is a classifying functor for unitary elements of any (unital) $\AF$-algebra. In what follows, we focus on the uniqueness part of the conjecture. As was done in the case of positive elements in \cite{RS09}, we first picture a unitary element $u$ of a $\CatCa$-algebra $A$ as a $^*$-homomorphism, and more specifically, as a unital $^*$-homomorphism $\varphi_u:\mathcal{C}(\T)\longrightarrow A$. We then use the $\Cu$-semimetric associated to $2^\infty$ on $\Hom_{\Cu}(\Lsc(\T,\overline{\N}),\Cu(A))$ to partially show that the $\Cu$-semigroup is classifying $^*$-homomorphisms from $\mathcal{C}(\T)$ to any unital $\AF$-algebra $A$. \\

\vspace{-0,2cm}$\hspace{-0,34cm}\bullet\,\,\textbf{Unitary elements - Unital homomorphisms - Approximate unitary equivalence}.$\\
Let $A$ be a unital $\CatCa$-algebra. There is a one-to-one correspondence between the set of unitary elements of $A$, that we write $\mathcal{U}(A)$, and the set of unital $^*$-homomorphisms from $\mathcal{C}(\T)$ to $A$, that we write $\Hom_{\CatCa,1}(\mathcal{C}(\T),A)$. Let $u$ be a unitary element of $A$. We define
\[
\begin{array}{ll}
	\hspace{-0,9cm}\varphi_u:\mathcal{C}(\T)\longrightarrow A\\
	\id_\T\longmapsto u
\end{array}
\] 
In other words, for any $f\in \mathcal{C}(\T)\supseteq C(\spectrum(u))$, we define $\varphi_u(f):=f(u)$, where $f(u)$ is obtained by functional calculus. Therefore, one can construct the following Set bijection: 
\[
\begin{array}{ll}
	\varphi:\mathcal{U}(A)\simeq \Hom_{\CatCa,1}(\mathcal{C}(\T),A)\\
	\hspace{1cm}u\longmapsto \varphi_u
\end{array}
\] 
Also, we recall that two $^*$-homomorphisms $\phi,\psi:A\longrightarrow B$ between two $\CatCa$-algebras $A$ and $B$ are said to be \emph{approximately unitarily equivalent} and we write $\phi\sim_{aue}\psi$, if there exists a sequence of unitary elements $(w_n)_n$ in $B^\sim$ such that $\Vert w_n\phi(x)w_n^*-\psi(x)\Vert\underset{n\infty}\longrightarrow 0$, for any $x\in A$. 
Similarly, two unitary elements $u,v$ of a unital $\CatCa$-algebra $B$ are said to be \emph{approximately unitarily equivalent}, and again we write $u\sim_{aue}v$, if there exists a sequence of unitary elements $(w_n)_n$ of $B$ such that $\Vert w_nuw_n^*-v\Vert\underset{n\infty}\longrightarrow 0$. 
It can be shown that $u\sim_{aue}v$ if and only if $\varphi_u\sim_{aue}\varphi_v$.

The conjecture states that the functor $\Cu$ is classifying $^*$-homomorphisms from $\mathcal{C}(\T)$ to $A$, for any $\AF$-algebra $A$. That is, for any $\Cu$-morphism $\alpha:\Cu(\mathcal{C}(\T))\longrightarrow \Cu(A)$, there exists a $^*$-homomorphism $\varphi:\mathcal{C}(\T)\longrightarrow A$, unique up to approximate unitary equivalence, such that $\Cu(\varphi)=\alpha$. Under the latter identification, we may abuse the language and say that \emph{$\Cu$ classifies unitary elements of $A$}. 

We are aiming for the uniqueness part of the theorem. To do so, we will first focus on the finite dimensional case. Subsequently, the $\AF$-case will follow as a consequence the lemma below. 

\begin{lma}
\label{thm:liftAF}
Let $\mathcal{B}_{2^\infty}=(\Lambda_n,\epsilon_n)_n$ be the finite uniform basis of $\Lsc(\T,\overline{\N})$ associated to $2^\infty$. Consider an inductive sequence $(S_i,\sigma_{ij})_{i\in \N}$ in the category $\Cu$ and its direct limit $(S,\sigma_{i\infty})_{i\in \N}$. 

Let $\alpha,\beta:\Lsc(\T,\overline{\N})\longrightarrow S$ be $\Cu$-morphisms that factorize through a finite stage. (In the sense that there exist $i\in\N$ and $\Cu$-morphisms $\alpha_i,\beta_i:\Lsc(\T,\overline{\N})\longrightarrow S_i$ such that $\alpha=\sigma_{i\infty}\circ \alpha_i$ and $\beta=\sigma_{i\infty}\circ \beta_i$.)\\ If $\alpha\underset{\Lambda_n}{\approx}\beta$ for some $n\in\N$, then there exists $j\geq i$ such that $\sigma_{ij}\circ\alpha_i\underset{\Lambda_{n-1}}{\approx}\sigma_{ij}\circ\beta_i$.
\end{lma}

\begin{proof}
Let $g',g\in\Lambda_{n-1}$ such that $g'\ll g$. By \autoref{lma:ccelmnt}, we can find an element $h\in\Lambda_n$ such that $g'\ll h\ll g$. Since $\alpha\underset{\Lambda_n}{\approx}\beta$, we obtain that \[\sigma_{i\infty}\circ\alpha_i(h),\,\sigma_{i\infty}\circ\beta_i(h)\ll \sigma_{i\infty}\circ\alpha_i(g),\,\sigma_{i\infty}\circ\beta_i(g).\]
Applying (L2) of \autoref{prop:caralimicu}, we deduce that there exists $j\geq i$ such that \[\sigma_{ij}\circ\alpha_i(g'),\,\sigma_{ij}\circ\beta_i(g')\ll \sigma_{ij}\circ\alpha_i(g),\,\sigma_{ij}\circ\beta_i(g).\]
Observe that $\Lambda_{n-1}$ is a finite set. Thus there exists $l\in\N$ big enough such that the above inequalities apply to any $g',g\in\Lambda_{n-1}$ such that $g'\ll g$. Equivalently, there exists $l\in\N$ big enough such that $\sigma_{il}\circ\alpha_i\underset{\Lambda_{n-1}}{\approx}\sigma_{il}\circ\beta_i$.
\end{proof}

The next step is to deal with the finite dimensional case. We will make use of a graph theory theorem, known as the \emph{Hall's marriage theorem}.

\begin{dfn}
A \emph{bipartite graph} is a graph whose vertices can be divided in two disjoint sets $X,Y$ such that every edge connects a vertex of $X$ to one of $Y$. We often write $G=(X+Y,E)$.
\end{dfn}

\begin{dfn}
Let $G=(E^0,E^1)$ be a graph. A \emph{matching} is a subset $F\subseteq E^1$ such that no two elements of $F$ share an endpoint. That is, any vertex of $G$ is an endpoint of at most one edge of $F$.

Let $G:=(X+Y,E)$ be a finite bipartite graph. By an \emph{$X$-saturating matching}, we refer to any matching that covers every vertex in $X$.
\end{dfn}

\begin{thm}\emph{(Hall's marriage theorem.)}
\label{thm:marriage}
Let $G:=(X+Y,E)$ be a finite bipartite graph with bipartite sets $X$ and $Y$. Let $W\subseteq X$. We define $n_{G}(W):=\bigcup\limits_{w\in W}\{y\in Y\mid (w,y)\in E\}$. In other words, $n_{G}(W)$ is the set of vertices in $Y$ that are linked with some $w$ in $W$. Then the following are equivalent: 

(i) There exists an $X$-saturating matching. 

(ii) For any $W\subseteq X$, $\card(W)\leq\card(n_G(W))$.
\end{thm}

\begin{thm}
\label{thm:liftFD}
Let $B$ be any finite dimensional $\CatCa$-algebra. Let $\mathcal{B}_{2^\infty}=(\Lambda_n,\epsilon_n)_n$ be the finite uniform basis of $\Lsc(\T,\overline{\N})$ associated to $2^\infty$.
Let $u,v$ be two unitary elements of $B$ such that $\Cu(\varphi_{u})\underset{\Lambda_n}{\approx}\Cu(\varphi_{v})$ for some $n\in\N$. 

Then there exists a unitary element $w$ in $B$ such that  $\Vert wuw^*-v\Vert <1/2^{n-1}$.
\end{thm}

\begin{proof}
We assume that $B\simeq M_l(\mathbb{C})$ for some $l\in\N$ and the general case will follow as a consequence. 

Let $X:=\spectrum(u),Y:=\spectrum(v)$ be the respective spectra of $u,v$ and let \[ E:=\{(x,y)\in X\times Y \mid\Vert x-y\Vert<2/2^n\}.\] Consider the bipartite graph $G:=(X+Y,E)$. Observe that $X$ and $Y$ correspond to the eigenvalues of $u$ and $v$ respectively. Thus, both sets are finite subsets of $\T$ of cardinal $l$. The general idea is to prove that (ii) of \autoref{thm:marriage} holds for $G$ in order to deduce that there exists an injective map $\sigma:X\lhook\joinrel\longrightarrow Y$ such that $\Vert x-\sigma(x)\Vert< 2/2^n$, for any $x\in Y$. We then conclude that $\sigma$ is a bijection using that $\card (X)=\card (Y)=l$.

 Let $\Omega$ be a subset of $X$. It is easy to see that there exists a minimal element $g_\Omega\in\Lambda_n$ such that $\Omega\subseteq \supp g_\Omega$, in the sense that for any other $g'_\Omega\in\Lambda$ whose spectrum contains $\Omega$, then $g_\Omega\leq g'_\Omega$. As done in previous proofs, we can assume without loss of generality that the support of $g_\Omega$ is a connected open set of $\T$, since it has finitely many connected components. 
We know that there exist $0\leq l,r\leq 2^n$ such that $\supp g_\Omega=U_{l}\cup(\bigcup\limits_{k=l+1}^{r-1}\overline{U_{k}})\cup U_{r}$. Furthermore, the indicator map of the open set $U_{l-1}\cup\overline{\supp g_\Omega}\cup U_{r+1}$, that we denote by $h_\Omega$, is the least element of $\{h\in \Lambda_n\mid g_\Omega\ll h\}$. (Convention $U_{2^n+1}:=U_1$ and $U_{0}:=U_{2^n}$.) From hypothesis, we know that 
\[
\Cu(\varphi_u)(g_\Omega),\Cu(\varphi_v)(g_\Omega)\ll \Cu(\varphi_u)(h_\Omega),\Cu(\varphi_v)(h_\Omega).
\]
Since $\sup\limits_{y\in \supp h_\Omega}d(y,\supp g_\Omega)= 2/2^n$, we have that 
\[\card(\Omega)=\Cu(\varphi_u)(g_\Omega)\leq \Cu(\varphi_v)(h_\Omega)\leq \card(n_G(\Omega))\]
from which we obtain that $\card(\Omega)\leq\card(n_G(\Omega))$, for any $\Omega\subset X$. Applying \autoref{thm:marriage}, we conclude that there exists a bijection $\sigma:X\simeq Y$ such that $\Vert \sigma - \id\Vert < 2/2^n$, which ends the proof.

\end{proof}

\begin{cor}
Let $A$ be a unital $\AF$-algebra and let $\varphi_{u},\varphi_{v}:\mathcal{C}(\T)\longrightarrow A$ be unital $^*$-homomorphisms. If $\Cu(\varphi_u)=\Cu(\varphi_v)$, then $\varphi_u$ and $\varphi_v$ are approximately unitarily equivalent.
\end{cor}

\begin{proof}
Let $(A_n,\phi_{nm})_n$ be an inductive system of finite dimensional $\CatCa$-algebras whose inductive limit is $A$. 
Let $u,v$ be unitary elements of $\mathcal{U}(A)$ and consider $\varphi_u,\varphi_v:\mathcal{C}(\T)\longrightarrow A$ their respective corresponding $^*$-homomorphisms. Let us write $\alpha:=\Cu(\varphi_u)$ and $\beta:=\Cu(\varphi_v)$ and assume that $\alpha=\beta$. We ought to show that $u$ and $v$ are approximately unitarily equivalent. 

Let $n\in\N$. We are going to find a unitary element $w\in A$ such that $\Vert wuw^*-v\Vert \leq 1/2^{n-2}$ from which corollary will follow. It is well-known that $\underset{n\in\N}{\cup} \phi_{n\infty}(A_n)$ is dense in $A$, therefore we can find unitary elements $u_i,v_i$ in some $A_i$ such that $\Vert\phi_{i\infty}(u_i)-u\Vert \leq 1/2^{n+4}$ and $\Vert\phi_{i\infty}(v_i)-v\Vert \leq 1/2^{n+5}$. Write $\alpha_i:=\Cu(\phi_{i\infty}\circ\varphi_{u_i})$ and $\beta_i:=\Cu(\phi_{i\infty}\circ\varphi_{v_i})$. We have that 
\begin{align*}
dd_{\Cu,\mathcal{B}_{2^\infty}}(\alpha_i,\beta_i)&\leq 2 ( dd_{\Cu,\mathcal{B}_{2^\infty}}(\alpha_i,\alpha)+dd_{\Cu,\mathcal{B}_{2^\infty}}(\alpha,\beta_i))\\
&\leq  2 dd_{\Cu,\mathcal{B}_{2^\infty}}(\alpha_i,\alpha)+4(dd_{\Cu,\mathcal{B}_{2^\infty}}(\alpha,\beta)+dd_{\Cu,\mathcal{B}_{2^\infty}}(\beta,\beta_i))\\
&\leq 2/2^{n+4}+4/2^{n+5}\\
&\leq 1/2^{n+2}.
\end{align*}
It is not hard to check that the latter implies $\alpha_i\underset{\Lambda_{n+1}}{\approx}\beta_i$. Using \autoref{thm:liftAF}, we can find some $j\geq i$ such that $\Cu(\phi_{ij}\circ\varphi_{u_i})\underset{\Lambda_{n}}{\approx}\Cu(\phi_{ij}\circ\varphi_{v_i})$. Let us rewrite $u_j:=\phi_{ij}(u_i)$ and $v_j:=\phi_{ij}(v_i)$. By \autoref{thm:liftFD}, there exists a unitary element $w$ in $A_j$ such that $\Vert wu_jw^* - v_j\Vert < 1/2^{n-1}$. Finally, we know that $\phi_{j\infty}$ is a contraction and hence we compute
\begin{align*}
\Vert \phi_{j\infty}(w)u\phi_{j\infty}(w)^*-v\Vert&\leq \Vert \phi_{j\infty}(w)[u-\phi_{j\infty}(u_j)]\phi_{j\infty}(w)^*\Vert+\Vert \phi_{j\infty}(w u_j w^*)-\phi_{j\infty}(v_j)\Vert\\&\,\,\,\,\,\,+\Vert \phi_{j\infty}(v_j)-v\Vert\\
&\leq 1/2^{n+4}+1/2^{n-1}+1/2^{n+5}\\
&\leq 1/2^{n-2}.
\end{align*}
\end{proof}

\end{document}